\documentclass[12pt]{zcf} 

\usepackage{iopams,amsthm}  
\begin{document}
\theoremstyle{plain} 
\newtheorem{theorem}{\bf Theorem}[section]
\newtheorem{lemma}[theorem]{\bf Lemma}
\newtheorem{corollary}[theorem]{\bf Corollary}
\newtheorem{proposition}[theorem]{\bf Proposition}
\theoremstyle{definition} 
\newtheorem{definition}[theorem]{\sc Definition}
\newtheorem{remark}[theorem]{\sc Remark}
\newtheorem{example}[theorem]{\sc Example}
%
\renewcommand{\proofname}{\bf Proof.}
\mathchardef\varTheta="0102
\mathchardef\varLambda="0103
\mathchardef\varGamma="0100
\mathchardef\varXi="0104
\mathchardef\varPi="0105
\mathchardef\varPsi="0109
\mathchardef\varSigma="0106
\mathchardef\varPhi="0108
\mathchardef\varOmega="010A
\font\tensans=cmss10
\def\bbR{\mathrm{I\!R}}
\def\bbL{\hbox{$I$\hskip-4.1pt$L$}}
\def\rto{\mathbf{R}\hskip-.5pt^2}
\def\rtr{\mathbf{R}\hskip-.7pt^3}
\def\rn{{\mathbf{R}}^{\hskip-.6ptn}}
\def\rk{{\mathbf{R}}^{\hskip-.6ptk}}
\def\hyp{\hskip.5pt\vbox
{\hbox{\vrule width3ptheight0.5ptdepth0pt}\vskip2.2pt}\hskip.5pt}
\def\xy{\hskip-1.2pt\overrightarrow{\hskip1.2ptxy\hskip1.2pt}\hskip-1.2pt}
\def\tzk{{T\hskip-3pt_z\hskip-.6ptK}}
\def\tzm{{T\hskip-3pt_z\hskip-.6ptM}}
\def\tzz{{T\hskip-3pt_z\hskip-.6ptZ}}
\def\txk{{T\hskip-3pt_x\hskip-.6ptK}}
\def\txm{{T\hskip-3pt_x\hskip-.6ptM}}
\def\txn{{T\hskip-2.7pt_x\hskip-.9ptN}}
\def\tzn{{T\hskip-3pt_z\hskip-.9ptN}}
\def\tazn{{T^*_{\hskip-3ptz}N}}
\def\tzp{{T\hskip-3pt_z\hskip-.9ptP}}
\def\tm{{T\hskip-.3ptM}}
\def\tn{{T\hskip-.3ptN}}
\def\tam{{T^*\!M}}
\def\mppp{\hbox{$-$\hskip1pt$+$\hskip1pt$+$\hskip1pt$+$}}
\def\mpdp{\hbox{$-$\hskip1pt$+$\hskip1pt$\dots$\hskip1pt$+$}}
\def\mmpp{\hbox{$-$\hskip1pt$-$\hskip1pt$+$\hskip1pt$+$}}
\def\mmmp{\hbox{$-$\hskip1pt$-$\hskip1pt$-$\hskip1pt$+$}}
\def\pppp{\hbox{$+$\hskip1pt$+$\hskip1pt$+$\hskip1pt$+$}}
\def\mpmp{\hbox{$-$\hskip1pt$\pm$\hskip1pt$+$}}
\def\mpmpp{\hbox{$-$\hskip1pt$\pm$\hskip1pt$+$\hskip1pt$+$}}
\def\mmpmp{\hbox{$-$\hskip1pt$-$\hskip1pt$\pm$\hskip1pt$+$}}
\def\hs{\hskip.7pt}
\def\hh{\hskip.4pt}
\def\nh{\hskip-.7pt}
\def\nnh{\hskip-1pt}
\def\ve{\varepsilon}

\voffset=-20pt\hoffset=12pt  

\title[Zeros of conformal fields]{Zeros of conformal fields in 
any metric signature}

\author{Andrzej Derdzinski}

\address{Department of Mathematics, 
The Ohio State University, 
Columbus, OH 43210, 
USA}

\ead{\mailto{andrzej@math.ohio-state.edu}}
\begin{abstract}
The connected components of the zero set of any conformal vector field, in a 
pseudo-Riemannian manifold of arbitrary signature, are shown to be totally 
umbilical conifold varieties, that is, smooth submanifolds except possibly 
for some quadric singularities. The singularities occur only when the metric 
is indefinite, including the Lorentzian case. This generalizes an analogous 
result in the Riemannian case, due to Belgun, Moroianu and Ornea (2010).
\end{abstract}

\maketitle

\setcounter{theorem}{0}
\renewcommand{\thetheorem}{\Alph{theorem}}
\section{Introduction}\label{intr}
A vector field $\,v\,$ on a pseu\-\hbox{do\hs-}\hskip0ptRiem\-ann\-i\-an 
manifold $\,(M,g)\,$ of dimension $\,n\ge2\,$ is called {\it con\-for\-mal\/} 
if, for some function $\,\phi:M\to\bbR$,
\begin{equation}\label{lvg}
\pounds_vg\,\,=\,\,\phi\hs g\hs,\hskip6pt\mathrm{that\ is,\ in\ 
coordinates,}\hskip5ptv_{j,\hh k}\nh+v_{\hh k,\hs j}\nh=\hs\phi\hs g_{jk}\hh.
\end{equation}
One then obviously has $\,\mathrm{div}\hskip2ptv=n\hs\phi/2$. The class of 
con\-for\-mal vector fields on $\,(M,g)\,$ includes {\it Kil\-ling fields\/} 
$\,v$, characterized by (\ref{lvg}) with $\,\phi=0$.

Kobayashi \cite{kobayashi} showed that, for any Kil\-ling vector field $\,v\,$ 
on a Riemannian manifold $\,(M,g)$, the connected components of the zero set 
of $\,v\,$ are mutually isolated totally geodesic sub\-man\-i\-folds of even 
co\-di\-men\-sions. Assuming compactness of $\,M\nh$, Blair \cite{blair} 
established an analogue of Ko\-ba\-ya\-shi's theorem for con\-for\-mal 
vector fields, in which the word `geodesic' is replaced by `umbilical' and 
the co\-di\-men\-sion clause is relaxed for one\hh-\hskip0ptpoint connected 
components. Very recently, Belgun, Moroianu and Ornea 
\cite{belgun-moroianu-ornea} proved that Blair's conclusion remains valid in 
the noncompact case.

It is natural to ask what happens when the metric $\,g\,$ is indefinite. 
Questions about the structure of con\-for\-mal fields arise in connection with 
some known open problems, such as those related to the 
pseu\-\hbox{do\hs-}\hskip0ptRiem\-ann\-i\-an Lich\-ne\-ro\-wicz conjecture 
\cite{frances}.

The result of Belgun, Moroianu and Ornea \cite{belgun-moroianu-ornea}, 
mentioned above, becomes false when repeated verbatim for indefinite metrics: 
even in pseu\-\hbox{do\hs-}\hskip0ptEuclid\-e\-an spaces, connected components 
of the zero set of a con\-for\-mal vector field may have quadric 
singularities (see Example~\ref{gensg} below). Such singularities, however, 
are the worst that can occur, aside from the fact that the codimension 
restriction has to be modified as well, cf.\ Example~\ref{gensg}. More 
precisely, the following theorem is proved in Section~\ref{cp}. (A set in a 
vector space is called {\it star-shap\-ed\/} if it is a union of line segments 
emanating from $\,0$.)
\begin{theorem}\label{tusbm}Let\/ $\,Z\,$ denote the zero set of a 
con\-for\-mal vector field\/ $\,v\,$ on a 
pseu\-\hbox{do\hs-}\hskip0ptRiem\-ann\-i\-an manifold\/ $\,(M,g)\,$ of 
dimension\/ $\,n\ge3$. Then every point\/ $\,z\in Z\,$ has a neighborhood\/ 
$\,\,U'$ in\/ $\,M\hs$ such that, for some star-shap\-ed neighborhood\/ 
$\,\,U\hs$ of\/ $\,0\,$ in\/ $\,\tzm\nh$, and some metric\/ $\,g\hh'$ on\/ 
$\,\,U'$ conformal to\/ $\,g$, the exponential mapping\/ $\,\mathrm{exp}\hh_z$ 
of\/ $\,g\hh'$ at\/ $\,z\,$ is defined on\/ $\,\,U\,$ and maps\/ $\,\,U\,$ 
dif\-feo\-mor\-phi\-cal\-ly onto\/ $\,\,U'\nnh$, while\/ 
$\,Z\cap U'\nh=\hh\mathrm{exp}\hh_z[E\cap U\hh]\,$ for\/ $\,E\subseteq\tzm\,$ 
which is
\begin{enumerate}
  \def\theenumi{{\rm\roman{enumi}}}
\item[{\rm(a)}] a vector subspace of\/ $\,\tzm\nh$, or
\item[{\rm(b)}] the set of all null vectors in a vector subspace\/ 
$\,H\subseteq\tzm\nh$.
\end{enumerate}
The singular subset\/ $\,\Delta\,$ of\/ $\,Z\cap U'$ equals\/ 
$\,\hh\mathrm{exp}\hh_z[H\cap H^\perp\nnh\cap U\hh]\,$ in case\/ 
{\rm(ii)}, if the metric restricted to\/ $\,H\,$ is not sem\-i\-def\-i\-nite, 
and\/ $\,\Delta=$\hskip3.5pt{\rm\O}\hskip4.5ptotherwise. The connected 
components of\/ $\,(Z\cap U')\smallsetminus\Delta\,$ are totally umbilical 
sub\-man\-i\-folds of\/ $\,(M,g)$, and their co\-di\-men\-sions are even 
unless\/ $\,\Delta=$\hskip3.5pt{\rm\O}\hskip4.5ptand\/ $\,Z\cap U'$ is a null 
totally geodesic sub\-man\-i\-fold of\/ $\,(M,g)$. In addition, 
$\,\mathrm{div}\hskip2ptv\,$ is constant along each connected component of\/ 
$\,Z$.
\end{theorem}
Remark~\ref{lorcs} discusses the meaning of Theorem~\ref{tusbm} in the 
Lo\-rentz\-i\-an case.

Theorem~\ref{tusbm} does not extend to dimension $\,2$. In the metric 
signature \hbox{$-$\hskip1pt$+$} the zero set of a con\-for\-mal field $\,v\,$ 
may be quite pathological (Example~\ref{srfcs}), even though on a {\it 
Riemannian\/} surface $\,(M,g)\,$ such $\,v\,$ is locally hol\-o\-mor\-phic, 
and so its zero set is discrete or equal to $\,M\nh$, cf.\ 
\cite{belgun-moroianu-ornea}.

The argument in Sections~\ref{cz} --~\ref{cp}, leading to Theorem~\ref{tusbm}, 
concentrates -- just as Belgun, Moroianu and Ornea did in 
\cite{belgun-moroianu-ornea} -- on the case where a con\-for\-mal vector field 
$\,v\,$ on a (pseu\-\hbox{do\hs-\nh})\hskip0ptRiem\-ann\-i\-an manifold 
$\,(M,g)\,$ has a zero at a point $\,z\in M\,$ satisfying one of the following 
two conditions, with $\,\phi\,$ as in (\ref{lvg}) and $\,\nabla\nh\phi_z$ 
denoting its gradient at $\,z\hh$:
\begin{equation}\label{cnd}
\begin{array}{rl}
\mathrm{a)}&\hskip3pt\phi(z)\hs\ne\hs0\hs,\\
\mathrm{b)}&\hskip3pt\phi(z)\,=\,0\hskip8pt\mathrm{and}\hskip8pt
\nabla\nh\phi_z\hs\notin\hs\nabla\nh v_z(\tzm)\hs.
\end{array}
\end{equation}
Here $\,\nabla\nh v_z(\tzm)\,$ is the image of 
$\,\nabla\nh v_z:\tzm\to\tzm\nh$, the value at $\,z\,$ of the covariant 
derivative $\,\nabla\nh v\,$ treated as the bundle morphism 
$\,\nabla\nh v:\tm\to\tm\nh\,$ which sends each vector field $\,w\,$ to 
$\,\nabla_{\!w}v$.

The use of (\ref{cnd}.a) -- (\ref{cnd}.b) is crucial in view of the following 
result of Beig \cite{beig}.
\begin{theorem}\label{escnf}For a con\-for\-mal vector field\/ $\,v\,$ on a 
pseu\-\hbox{do\hs-}\hskip0ptRiem\-ann\-i\-an manifold\/ $\,(M,g)\,$ with\/ 
$\,\dim M\ge3\,$ and a zero\/ $\,z\,$ of $\,v$, the following two conditions 
are equivalent\/{\rm:}
\begin{enumerate}
  \def\theenumi{{\rm\roman{enumi}}}
\item[{\rm(i)}] $z\,$ has a neighborhood\/ $\,\,U'$ such that\/ $\,v\,$ 
restricted to\/ $\,\,U'$ is a Kil\-ling field on\/ $\,(U'\nnh,g\hh'\hh)$, 
where\/ $\,g\hh'$ is some metric on\/ $\,\,U'$ con\-for\-mal to\/ $\,g$,
\item[{\rm(ii)}] $\phi(z)=0\,$ and\/ 
$\,\nabla\nh\phi_z\in\nabla\nh v_z(\tzm)$, that is, neither\/ 
{\rm(\ref{cnd}.a)} nor\/ {\rm(\ref{cnd}.b)} holds at\/ $\,z$.
\end{enumerate}
\end{theorem}
\begin{proof}See \cite{capocci}.
\end{proof}
A point $\,z\in M\,$ is said to be {\it essential\/} 
\cite{belgun-moroianu-ornea} for a con\-for\-mal vector field $\,v\,$ on 
$\,(M,g)\,$ if condition (i) in Theorem~\ref{escnf} fails to hold. Thus, by 
Theorem~\ref{escnf}, essential zeros of $\,v\,$ are precisely those zeros at 
which (\ref{cnd}.a) or (\ref{cnd}.b) is satisfied. On the other hand, points 
where $\,v\ne0\,$ are never essential, cf.\ the lines preceding 
Lemma~\ref{krnll}.

In proving Theorem~\ref{tusbm} we are allowed, by Theorem~\ref{escnf}, to make 
the additional assumption that (\ref{cnd}.a) or (\ref{cnd}.b) holds. In fact, 
if one has neither (\ref{cnd}.a) nor (\ref{cnd}.b), Theorem~\ref{escnf} 
reduces the problem to studying the zero set of a Kil\-ling field, which is 
always lin\-e\-ar\-ized by normal coordinates. Assertion (a) of 
Theorem~\ref{tusbm} then follows, for 
$\,E=\mathrm{Ker}\hskip1.7pt\nabla\nh v_z$, with $\,g\hh'$ chosen as in 
Theorem~\ref{escnf}(i). (See Section~\ref{cp}.) 

On the other hand, if one of conditions (\ref{cnd}.a) -- (\ref{cnd}.b) is 
satisfied, case (b) in Theorem~\ref{tusbm} is a direct consequence of the 
following result, proved in Sections~\ref{ca} --~\ref{cb}:
\begin{theorem}\label{nbzrs}Let\/ $\,Z\,$ be the zero set of a con\-for\-mal 
vector field\/ $\,v\,$ on a pseu\-\hbox{do\hs-}\hskip0ptRiem\-ann\-i\-an 
manifold\/ $\,(M,g)\,$ of dimension\/ $\,n\ge3$. If\/ $\,z\in Z$, while\/ 
$\,\mathrm{exp}\hh_z$ denotes the exponential mapping of\/ $\,g\,$ at\/ 
$\,z$, the function\/ $\,\phi\,$ in\/ {\rm(\ref{lvg})} has one of the 
properties\/ {\rm(\ref{cnd}.a)} -- {\rm(\ref{cnd}.b)}, and\/ $\,\,U\hs$ is a 
sufficiently small star-shap\-ed neighborhood of\/ $\,0\,$ in\/ $\,\tzm\,$ 
mapped by\/ $\,\mathrm{exp}\hh_z$ dif\-feo\-mor\-phi\-cal\-ly onto a 
neighborhood\/ $\,\,U'$ of\/ $\,z\,$ in\/ $\,M\nh$, then\/ 
$\,Z\cap U'\nh=\hh\mathrm{exp}\hh_z[\hh C\cap H\cap U\hh]\,$ for\/ 
$\,H=\hs\mathrm{Ker}\hskip1.7pt\nabla\nh v_z
\cap\hs\mathrm{Ker}\hskip2.2ptd\hh\phi_z\nh\subseteq\tzm\,$ and the null 
cone\/ $\,C=\{u\in\tzm:g_z(u,u)=0\hs\}$.
\end{theorem}
The paper is organized as follows. Sections~\ref{ms} --~\ref{is} contain 
preliminary material, including Theorem~\ref{tgdir} derived from the 
\hbox{Morse\hs-}\hskip0ptBott lemma (more on which below). The three lemmas in 
Section~\ref{cg}, which deal with the behavior of con\-for\-mal fields along 
null geodesics, are then used in Section~\ref{fi} to establish the relation 
$\,\mathrm{exp}\hh_z[\hh C\cap H\cap U\hh]\subseteq Z\cap U'\nnh$, one of the 
two opposite inclusions constituting the equality 
$\,Z\cap U'\nh=\hh\mathrm{exp}\hh_z[\hh C\cap H\cap U\hh]\,$ required in 
Theorem~\ref{nbzrs}. The proof of the remaining inclusion is split into 
Sections~\ref{ca} and~\ref{cb}, corresponding to two separate cases, 
(\ref{cnd}.a) and (\ref{cnd}.b). In the former, limiting properties of 
geodesic segments joining points of 
$\,\mathrm{exp}\hh_z[\hh C\cap H\cap U\hh]\,$ to other zeros of $\,v\,$ near 
$\,z\,$ are used to conclude that the other zeros cannot lie arbitrarily close 
to $\,z$. A similar argument provides a part of the proof in the latter case: 
phrased as Lemma~\ref{tadir}, it shows that nearby zeros at which 
$\,\phi\ne0\,$ would lead to connecting limits, in the sense of 
Section~\ref{cl}, for certain subsets of $\,Z\,$ near $\,z$, which are 
contained in $\,H$, but not in the null\-space of $\,H$. The final step is 
provided by Theorem~\ref{tgdir}, which states, first, that the existence of 
such connecting limits would contradict the algebraic structure of the second 
covariant derivative of $\,v\,$ at $\,z$, and, secondly, that nearby zeros 
with $\,\phi=0\,$ must all lie in $\,\mathrm{exp}\hh_z[\hh C\cap H\cap U\hh]$.

\renewcommand{\thetheorem}{\thesection.\arabic{theorem}}
\section{Manifolds and sub\-man\-i\-folds}\label{ms}
Unless stated otherwise, manifolds and sub\-man\-i\-folds are connected, 
sub\-man\-i\-folds carry the subset topology, while tensor fields and mappings 
are, by definition, of class $\,C^\infty\nnh$. By vec\-tor-val\-ued functions 
we mean mappings into vector spaces, with the latter always assumed to be 
fi\-\hbox{nite\hh-}\hskip0ptdi\-men\-sion\-al and real.

Given a vec\-tor-val\-ued function $\,\beta\,$ on $\,I\times K$, where $\,K\,$ 
is a manifold and $\,I\subset\bbR\,$ an interval containing $\,0$, the 
New\-ton-Leib\-niz formula and integration by parts yield
\begin{equation}\label{bei}
\begin{array}{rl}
\mathrm{i)}&\hskip3pt\beta(s,y)\,=\,\beta(0,y)\,
+\,s\int_0^1\beta_s^{\phantom i}(ts,y)\,dt\hs,\\
\mathrm{ii)}&\hskip3pt\beta(s,y)\,=\,\beta(0,y)\,
+\,\beta_s^{\phantom i}(0,y)\hs s\,
+\,s^2\int_0^1(1-t)\beta_{ss}^{\phantom i}(ts,y)\,dt
\end{array}
\end{equation}
for $\,s\in I\hs$ and $\,y\in K$, with 
$\,(\hskip2.3pt)_s^{\phantom i}\nh=\hh\partial/\partial\hh s$. In particular, 
$\,\beta\,$ is smoothly divisible by the projection function 
$\,(s,y)\mapsto s\,$ if $\,\beta=0\,$ whenever $\,s=0$. Similarly, for a 
vec\-tor-val\-ued function $\,\varPsi:U\to\mathcal{T}\hs$ on a neighborhood 
$\,\,U\,$ of a point $\,z\,$ in a vector space $\,W\nnh$, and any $\,x,y\,$ 
near $\,z\,$ in $\,W\nnh$, integrating $\,d\hs[\hh\varPsi(x+t(y-x))]/dt\,$ 
from $\,t=0\,$ to $\,t=1\,$ we obtain
\begin{equation}\label{fwe}
\varPsi(y)\,-\,\varPsi(x)\,=\,D\nh_{(x,y)}^{\phantom i}(y-x)\hs,\hskip18pt
\mathrm{with}\hskip9ptD_{(x,x)}^{\phantom i}\hs
=\,d\hs\varPsi_x^{\phantom i}\hh,
\end{equation}
where the function 
$\,(x,y)\mapsto 
D\nh_{(x,y)}^{\phantom i}\in\mathrm{Hom}\hs(W\nnh,\mathcal{T})\,$ is given by 
$\,D\nh_{(x,y)}^{\phantom i}
=\int_0^1d\hs\varPsi\nnh_{x+t(y-x)}^{\phantom i}\hskip1.7ptdt$.
\begin{lemma}\label{nzero}Let a vec\-tor-val\-ued function\/ $\,\beta\hh$ on 
a manifold $\,N$ vanish at all points of a 
co\-di\-men\-\hbox{sion\hh-}\hskip.7ptone sub\-man\-i\-fold\/ $\,K$.
\begin{enumerate}
  \def\theenumi{{\rm\alph{enumi}}}
\item[{\rm(a)}] If\/ $\,z\in K\,$ and\/ $\,d\hh\beta_z\nh\ne0$, then\/ 
$\,z\,$ has a neighborhood $\,\,U\,$ in $\,N$ such that\/ $\,\beta\ne0\,$ 
everywhere in\/ $\,\,U\smallsetminus K\nnh$.
\item[{\rm(b)}] If\/ $\,d\hh\beta\ne0\,$ everywhere in a set\/ 
$\,\varXi\subseteq K$, then for some open subset\/ $\,\,U\,$ of\/ $\,N$ 
containing\/ $\,\varXi\hh$ we have\/ $\,\beta\ne0\,$ at all points of\/ 
$\,\,U\smallsetminus K\nnh$.
\end{enumerate}
\end{lemma}
\begin{proof}Let us replace $\,K\hs$ with a smaller neighborhood of $\,z\,$ in 
$\,K$, if necessary, so as to identify a neighborhood of $\,z\,$ in 
$\,N\hs$ with $\,I\times K\,$ and $\,K\,$ with $\,\{0\}\times K$, for $\,I\,$ 
as in (\ref{bei}). Since $\,d\hh\beta_z\nh\ne0$, we have 
$\,\int_0^1\beta_s^{\phantom i}(0,y)\,dt=\beta_s^{\phantom i}(0,y)\ne0\,$ in 
(\ref{bei}.i), for $\,y=z$. This yields (a), while (a) obviously implies (b).
\end{proof}
Given a sub\-man\-i\-fold $\,K\,$ of a manifold $\,M\nh$, the {\it normal 
bundle\/} of $\,K\,$ is defined, as usual, to be the quotient vector bundle 
$\,\mathcal{N}=T\hskip-2.7pt_K^{\phantom i}\nnh M/T\nh K$, where 
$\,T\hskip-2.7pt_K^{\phantom i}\nnh M\,$ is the restriction 
of $\,\tm\,$ to $\,K$. A fixed tor\-sion\-free connection $\,\nabla\,$ on 
$\,M\,$ then gives rise to the {\it second fundamental form\/} of $\,K$, which 
is a section $\,b\,$ of $\,\mathrm{Hom}\hs([\tm]^{\odot2}\nnh,\mathcal{N})
=[\tam]^{\odot2}\nnh\otimes\mathcal{N}\,$ (in other words, 
$\,b_x:T_xK\times T_xK\to\mathcal{N}\nnh_x$ is, at every $\,x\in K$, 
bi\-lin\-e\-ar and symmetric). We have 
$\,b(\dot x,\dot x)=\pi\nabla_{\!\dot x}\dot x\,$ whenever $\,t\mapsto x(t)\,$ is a curve in $\,K$, with $\,\pi:\tm\to\mathcal{N}\,$ denoting the quotient projection.
\begin{lemma}\label{tgapt}Let\/ $\,b\,$ be the second fundamental form of a 
sub\-man\-i\-fold\/ $\,K\,$ in a manifold\/ $\,M\,$ endowed with a 
tor\-sion\-free connection $\,\nabla\nnh$.
\begin{enumerate}
  \def\theenumi{{\rm\roman{enumi}}}
\item[{\rm(i)}] $b(\dot x,\dot x)=0\,$ along any geodesic\/ 
$\,t\mapsto x(t)\,$ of\/ $\,\nabla\hs$ which is  contained in\/ $\,K$.
\item[{\rm(ii)}] If\/ $\,z\in M\nh$, a neighborhood\/ $\,\,U\,$ of\/ 
$\,0\,$ in\/ $\,\tzm\,$ is mapped by\/ $\,\mathrm{exp}\hh_z$ 
dif\-feo\-mor\-phi\-cal\-ly onto a neighborhood of\/ $\,z\,$ in\/ $\,M\nh$, 
and\/ $\,K\hs\,=\,\hs\mathrm{exp}\hh_z[\hh V\nh\cap U\hh]\,$ for a vector 
subspace\/ $\,V$ of\/ $\,\tzm\nh$, then $\,b_z\nh=0$.
\end{enumerate}
\end{lemma}
\begin{proof}Assertion (i) is obvious from the formula 
$\,b(\dot x,\dot x)=\pi\nabla_{\!\dot x}\dot x$, and (ii) from (i) for all the 
geodesics $\,x(t)=\mathrm{exp}\hh_z\hs tu\,$ with $\,u\in V\nnh$.
\end{proof}
When $\,b=0\,$ identically, $\,K\,$ is said to be {\it totally geodesic\/} 
relative to $\,\nabla\nnh$. If $\,\nabla\,$ is the Le\-vi-Ci\-vi\-ta 
connection of a pseu\-\hbox{do\hs-}\hskip0ptRiem\-ann\-i\-an metric $\,g\,$ on 
$\,M\,$ and $\,b=g_{\nh K}^{\phantom i}\nh\otimes u\,$ for some section $\,u\,$ 
of $\,\mathcal{N}\nh$. where $\,g_{\nh K}^{\phantom i}$ is the restriction of 
$\,g\,$ to $\,K$, one calls $\,K\,$ {\it totally umbilical\/} in $\,(M,g)$. 
This last property of $\,K\,$ is con\-for\-mal\-ly invariant, since
\begin{equation}\label{inv}
\mathrm{changing\ }\,g\,\mathrm{\ to\ }\,e^{-\tau}\hskip-2ptg\,
\mathrm{\ causes\ }\,b\,\mathrm{\ to\ be\ replaced\ by\ }
\,b\hskip1pt+g_{\nh K}^{\phantom i}\nh\otimes\pi\nabla\nh\tau/2\hh.
\end{equation}
A null sub\-man\-i\-fold of a pseu\-\hbox{do\hs-}\hskip0ptRiem\-ann\-i\-an 
manifold $\,(M,g)\,$ is totally umbilical if and only if it is totally 
geodesic. The class of (un\-pa\-ram\-e\-trized) null geodesics in $\,(M,g)\,$ 
is therefore a con\-for\-mal invariant.

\section{Differentials and Hess\-i\-ans}\label{dh}
As before, vec\-tor-val\-ued functions are mappings into 
fi\-\hbox{nite\hh-}\hskip0ptdi\-men\-sion\-al real vector spaces. For a fixed 
curve $\,t\mapsto x(t)\,$ in a manifold $\,M\,$ and a vec\-tor-val\-ued 
function $\,f\,$ on $\,M\nh$, we write
\begin{equation}\label{dfe}
\hskip2.8pt\dot{\hskip-3.3ptf}\,=\,d\hh[\hh f(x(t))]/dt\hs,\hskip12pt
\hskip3.3pt\ddot{\hskip-3.3ptf}\,=\,d^{\hs2}\hh[\hh f(x(t))]/dt^2.
\end{equation}
Given a vector bundle $\,\mathcal{E}\,$ over a manifold $\,M$, a section 
$\,\psi\,$ of $\,\mathcal{E}\,$ defined on an open set $\,\,U\subseteq M\nh$, 
and a point $\,z\in U\,$ at which $\,\psi_z=\hh0$, the {\it differential\/} 
of $\,\psi\,$ at $\,z\,$ is the linear operator 
$\,\partial\hh\psi_z:\tzm\to\mathcal{E}_z$ arising as the composition of 
the ordinary differential $\,d\psi_z:\tzm\to T_z\hh\mathcal{E}\,$ and the 
projection $\,T_z\hh\mathcal{E}\to\mathcal{E}_z$ coming from the natural 
identification $\,T_z\hh\mathcal{E}\,\approx\,\tzm\oplus\mathcal{E}_z$. (Here 
$\,z\in M\subseteq\mathcal{E}$, with $\,M\,$ treated as the zero section 
embedded in the total space $\,\mathcal{E}$, and $\,\psi\,$ viewed as a 
mapping $\,M\to\mathcal{E}$.) The components of $\,\partial\hh\psi_z$ 
relative to fixed local coordinates and a local trivialization of 
$\,\mathcal{E}$, defined around $\,z$, are $\,\partial_j\psi^a\nh$, 
so that $\,\partial\hh\psi_z=\nabla\psi_z$ for any connection $\,\nabla\,$ in 
$\,\mathcal{E}$.
\begin{example}\label{zrset}For $\,\mathcal{E}\nh,M\nh,\psi,\,U\,$ and $\,z\,$ 
as above, with $\,\psi_z=\hh0$, let $\,n=\dim M\,$ and 
$\,r=\hs\mathrm{rank}\hskip2.9pt\partial\hh\psi_z$. Then all zeros of 
$\,\psi\,$ near $\,z\,$ lie in some 
\hbox{$\,(n-r)\hh$-}\hskip0ptdi\-men\-sion\-al sub\-man\-i\-fold $\,N\hs$ of 
$\,M\,$ having the property that 
$\,\mathrm{Ker}\hskip2.2pt\partial\hh\psi_x\subseteq\txn\hs$ whenever 
$\,x\in N\hs$ and $\,\psi_x=\hh0$.

To construct such $\,N\nh$, we may start with an 
\hbox{$\,r$-}\hskip0ptdi\-men\-sion\-al real vector space $\,W\hs$ and a 
base-pre\-serv\-ing bundle mor\-phism $\,G\,$ from $\,\mathcal{E}\,$ into the 
product vector bundle $\,M\times W\nnh$. In other words, $\,G\,$ is a 
$\,W\nnh$-val\-ued $\,C^\infty$ function on the total space $\,\mathcal{E}\,$ 
and its restriction $\,G_x$ to the fibre $\,\mathcal{E}_x$ is linear for 
each $\,x\in M\nh$. We now choose $\,G\,$ so that $\,G_z$ sends the image 
$\,\partial\hh\psi_z(\tzm)\,$ iso\-mor\-phi\-cal\-ly onto $\,W\nnh$. The 
mapping $\,F:U\to W\hs$ defined by $\,F(x)=G_x\psi_x$ has, at any $\,x\in U\,$ 
with $\,\psi_x=\hh0$, the differential 
$\,d\hskip-.8ptF\hskip-1.7pt_x\nh=G_x\nh\circ\partial\hh\psi_x$. Applied to 
$\,x=z$, this shows that $\,F\,$ is a sub\-mer\-sion at $\,z\,$ and, making 
$\,\,U\,$ smaller if necessary, we can simply set 
$\,N\nh=\hs U\nh\cap\hh F^{-1}(0)$.
\end{example}
\begin{example}\label{tgctg}(a)\hskip6ptIn the case where 
$\,\mathcal{E}=\tm\,$ and $\,z\in M\,$ is a zero of a vector field $\,v\,$ 
defined on a neighborhood of $\,z$, the differential $\,\partial\hh v_z$ 
obviously coincides with the infinitesimal generator of the local flow of 
$\,v\,$ acting in $\,\tzm$.

(b)\hskip6ptIf $\,f:M\to W\hs$ is a vec\-tor-val\-ued function on a manifold 
$\,M\,$ and $\,d\hskip-.8ptf_{\nh z}\nh=0\,$ at a point $\,z\in M\nh$, the 
differential $\,\partial\hs d\hskip-.8ptf_{\nh z}$ of $\,d\hskip-.8ptf\,$ at 
$\,z\,$ is nothing else than the Hess\-i\-an of $\,f\,$ at $\,z$. Here 
$\,d\hskip-.8ptf\,$ is a section of the bundle $\,\mathcal{E}\hs$ of 
$\,W\nnh$-val\-ued $\,1$-forms on $\,M\nh$. We will use the fact that, for a 
curve $\,t\mapsto x(t)\,$ in $\,M\,$ with $\,x(0)=z$,
\begin{equation}\label{hss}
\hskip3.3pt\ddot{\hskip-3.3ptf}(0)\,\,
=\,\,\partial\hs d\hskip-.8ptf_{\nh z}(u,u)\hskip8pt\mathrm{(notation\ of\ 
(\ref{dfe})),}\hskip12pt\mathrm{where}\hskip6ptu\hs=\hs\dot x(0)\hs.
\end{equation}
\end{example}
\begin{remark}\label{hessn}Let $\,N\hs$ be the sub\-man\-i\-fold constructed 
in Example~\ref{zrset}, with the corresponding 
$\,\mathcal{E}\nh,M\nh,\psi,\,U\nh,z,n,r,G\,$ and $\,F\nnh$. Suppose that, 
in addition, $\,\mathcal{E}\,$ and $\,\tm\,$ are endowed with connections, 
of which the latter is tor\-sion\-free, and both are denoted by 
$\,\nabla\nnh$, while $\,\xi\,$ is a fixed section of the dual bundle 
$\,\mathcal{E}\hh^*\nnh$, and the function $\,Q:N\to\bbR\,$ is defined to be 
the restriction of $\,\xi(\psi)\,$ to $\,N\nnh$. Then
\begin{enumerate}
  \def\theenumi{{\rm\roman{enumi}}}
\item[{\rm(i)}] $d\hh Q\nh_x\nh=0\,$ at every $\,x\in N\hs$ with 
$\,\psi_x=\hh0\,$ and $\,\mathrm{rank}\hskip2.9pt\partial\hh\psi_x\nh=r\nh$, 
including $\,x=z$,
\item[{\rm(ii)}] in the case where $\,\xi_z$ vanishes on the image 
$\,\partial\hh\psi_z(\tzm)\,$ (that is, 
$\,\partial\hh\psi_z(\tzm)\subseteq\hs\mathrm{Ker}\hskip2.2pt\xi_z$), the 
Hess\-i\-an of $\,Q\,$ at $\,z\,$ is given by 
$\,\partial\hs d\hh Q_z(u,u)=\xi(\theta(u,u))\,$ for all 
$\,u\in\tzn\nh=\hs\mathrm{Ker}\hskip2.2pt\partial\hh\psi_z$, with 
$\,\theta\,$ denoting the second covariant derivative of $\,\psi\,$ relative 
to the two connections and their tensor product: 
$\,\theta(u,w)=[\nabla_{\!u}(\nabla\psi)]w\,$ whenever $\,u,w\in\tzm$.
\end{enumerate}
In fact, for a curve $\,t\mapsto x(t)\,$ in $\,N\hs$ we have 
$\,\dot Q=(\xi_a\psi^a)\dot{\,}=\dot x^{\hs j}(\xi_{a,\hh j}\psi^a\nh
+\xi_a\psi^a{}_{,\hh j}),$ (notation of (\ref{dfe}) and the lines preceding 
Example~\ref{zrset}, with commas standing for covariant derivatives). This 
gives (i), since for all $\,x\,$ in question the inclusion 
$\,\mathrm{Ker}\hskip2.2pt\partial\hh\psi_x\subseteq\txn\hs$ is an equality. 
If $\,x(0)=z\,$ and $\,\dot x(0)=u$, differentiating covariantly with respect 
to $\,t\,$ once again, at $\,t=0$, we obtain, from (\ref{hss}), 
$\,\partial\hs d\hh Q_z(u,u)=u^{\hs j}u^{\hs k}\xi_a(z)\psi^a{}_{,\hh jk}(z)$, 
as required; note that 
$\,\xi_a\psi^a{}_{,\hh j}\nh=0\,$ at $\,t=0\,$ 
(since $\,\partial\hh\psi_z(\tzm)\subseteq\hs\mathrm{Ker}\hskip2.2pt\xi_z$), 
and $\,\psi^a{}\nnh_{,\hs jk}\nh=\theta^{\hh a}_{kj}$.
\end{remark}

\section{Nor\-mal-co\-or\-di\-nate neighborhoods and rigid geodesics}\label{nc}
As before, a subset of a vector space is said to be {\it star-shap\-ed\/} if 
it is a union of line segments emanating from $\,0$.

For a fixed connection $\,\nabla$ on a manifold $\,M\nh$, a nontrivial 
$\,\nabla\nnh$-ge\-o\-des\-ic segment $\,\varGamma\hs$ with endpoints 
$\,y,x\,$ will be called {\it rigid\/} if there exists an open subset 
$\,\,U'$ of $\,M\,$ containing $\,\varGamma\hs$ such that $\,\varGamma\hs$ 
is the unique geodesic segment in $\,\,U'$ joining $\,y\,$ to $\,x$. By a {\it 
nor\-mal-co\-or\-di\-nate neighborhood\/} of a point $\,z\in M\,$ we mean any 
open set $\,\,U'\nh\subseteq M\,$ which is the 
$\,\mathrm{exp}\hh_z$-dif\-feo\-mor\-phic image of a star-shap\-ed 
neighborhood $\,\,U\,$ of $\,0\,$ in $\,\tzm\nh$, with $\,\mathrm{exp}\hh_z$ 
denoting the exponential mapping of $\,\nabla$ at $\,z$. Each point 
$\,x\in U'$ then is joined to $\,z\,$ by a unique 
$\,\nabla\nnh$-ge\-o\-des\-ic segment $\,\varGamma\hs$ contained in 
$\,\,U'$ (and so $\,\varGamma\hs$ is rigid).

Let $\,\tm\,$ be the total space of the tangent bundle of $\,M\nh$. Our 
convention is that, as a set, 
$\,\tm=\{(x,w):x\in M\hskip5pt\mathrm{and}\hskip5pt w\in\txm\}$. We identify 
$\,M\,$ and each tangent space $\,\txm\nh$, in the standard way, with subsets 
of $\,\tm\,$ (the zero section and the fibre $\,\{x\}\times\txm$), and say 
that a subset of $\,\tm\,$ is {\it radial\/} if its intersection with each 
$\,\txm\,$ is a (possibly empty) star-shaped set in $\,\txm\nh$. For a fixed 
connection $\,\nabla$ on $\,M\nh$, the formula 
$\,\mathrm{Exp}\hh(x,w)=(x,\hh\mathrm{exp}\hh_z\hs w)\,$ defines a mapping 
from a radial open sub\-man\-i\-fold of $\,\tm\nh$, containing the zero 
section, into $\,M\times M\nh$. In view of the inverse mapping theorem, 
$\,\mathrm{Exp}\,$ restricted to a suitable radial neighborhood 
$\,\varOmega\,$ of any point $\,(z,0)\,$ in the zero section is a 
dif\-feo\-mor\-phism onto a neighborhood $\,\varOmega\hh'$ of $\,(z,z)\,$ in 
$\,M\times M\nh$. We will call a nor\-mal-co\-or\-di\-nate neighborhood 
$\,\,U'$ of $\,z\,$ {\it sub\-con\-vex\/} if 
$\,\,U'\nnh\times U'\nh\subseteq\varOmega\hh'$ for some such $\,\varOmega\,$ 
and $\,\varOmega\hh'\nnh$. More precisely, we will treat $\,\varOmega\,$ as a 
``part of the structure'' of the sub\-con\-vex nor\-mal-co\-or\-di\-nate 
neighborhood $\,\,U'$ of $\,z$, so that, whenever 
$\,(x,y)=\mathrm{Exp}\hh(x,w)\in U'\nnh\times U'$ with $\,(x,w)\in\varOmega$, 
we may refer to the curve 
$\,[\hs0,1\hh]\ni t\mapsto\mathrm{Exp}\hh(x,tw)\,$ as {\it the\/} rigid 
geodesic segment in $\,M\,$ joining $\,x\,$ to $\,y$.
\begin{remark}\label{nlseg}If $\,\nabla\,$ is the Le\-vi-Ci\-vi\-ta connection 
of a pseu\-\hbox{do\hs-}\hskip0ptRiem\-ann\-i\-an metric $\,g\,$ on $\,M\nh$, 
the class of null geodesic segments in $\,(M,g)\,$ depends only on the 
underlying con\-for\-mal structure of $\,g$. (See the end of Section~\ref{ms}.) 
On the other hand, for any con\-for\-mal vector field $\,v\,$ on $\,(M,g)$, 
the local flow of $\,v\,$ consists of con\-for\-mal dif\-feo\-mor\-phisms. 
Consequently, if $\,v=0\,$ at both endpoints of a rigid nontrivial null 
geodesic segment $\,\varGamma\hs$ in $\,(M,g)$, then, due to uniqueness of 
$\,\varGamma\nnh$, the local flow of $\,v\,$ maps $\,\varGamma\,$ into itself.
\end{remark}

\section{Connecting limits and radial limit directions}\label{cl}
Suppose that $\,M\,$ is a manifold, $\,z\in M\nh$, and $\,L\,$ is a line 
through $\,0\,$ in $\,\tzm\nh$. Let us also fix a norm $\,|\hskip3pt|\,$ in 
$\,\tzm\,$ and a neighborhood $\,\,U\,$ of $\,0\,$ in $\,\tzm\,$ along with 
a dif\-feo\-mor\-phism $\,\varPsi:U\to U'$ onto a neighborhood $\,\,U'$ of 
$\,z\,$ in $\,M\,$ such that $\,\varPsi(0)=z\,$ and 
$\,d\hs\varPsi_0^{\phantom i}=\hh\mathrm{Id}:\tzn\to\tzn\nnh$. We call $\,L\,$ 
a {\it connecting limit\/} for a pair of sequences $\,x_j,y_j\in M$, 
$\,j=1,2,\ldots\hs$, both converging to $\,z\,$ and having 
$\,x_j\nh\ne y_j$ whenever $\,j\,$ is  sufficiently large, if, for all but 
finitely many $\,j$, and $\,u_j,w_j\nh\in U\,$ such that $\,\varPsi(u_j)=x_j$, 
$\,\varPsi(w_j)=y_j$, the limit of the sequence 
$\,(w_j\nh-u_j)/|\hh w_j\nh-u_j|\,$ exists and spans $\,L$.

For such $\,M\nh,z\,$ and $\,x_j,y_j$, neither $\,L\,$ itself nor the fact of 
its existence depends on the choice of $\,|\hskip3pt|\,$ and $\,\varPsi$. This 
is obvious for $\,|\hskip3pt|$, and for $\,\varPsi\,$ it amounts to the 
following claim: {\it if, in addition, $\,M\,$ is a neighborhood of\/ 
$\,z=0\,$ in a vector space\/ $\,W\nnh$, so that\/ $\,\tzm\nh=W\nnh$, and 
$\,(y_j\nh-x_j)/|y_j\nh-x_j|\to u\in W\hs$ as $\,j\to\infty$, then, for any 
dif\-feo\-mor\-phism\/ $\,\varPsi\,$ with the stated properties, 
$\,[\hh\varPsi(y_j)-\varPsi(x_j)]/|\hh\varPsi(y_j)-\varPsi(x_j)|\to u$.}

To verify the italicized statement, note that, for $\,x,y\,$ near $\,0\,$ in 
$\,W\nnh$, if one writes $\,x=x_j$, $\,y=y_j$, the assumption that 
$\,(y-x)/|\hh y-x|\to u\,$ gives, by (\ref{fwe}), 
$\,[\hh\varPsi(y)-\varPsi(x)]/|\hh y-x|
=D\nh_{(x,y)}^{\phantom i}[(y-x)/|\hh y-x|]\to D_{(0,0)}^{\phantom i}u=u$. 
Since $\,[\hh\varPsi(y)-\varPsi(x)]/|\hh y-x|\,$ tends to the 
$\,|\hskip3pt|$-unit vector $\,u$, so does the sequence 
$\,[\hh\varPsi(y)-\varPsi(x)]/|\hh\varPsi(y)-\varPsi(x)|\,$ obtained by 
normalizing $\,[\hh\varPsi(y)-\varPsi(x)]/|\hh y-x|$, as required.
\begin{remark}\label{submf}Given $\,M\nh,z,L\,$ as above, let 
$\,L\subset\tzm\,$ be the connecting limit for a pair of sequences 
$\,x_j,y_j$ with $\,x_j\nh\ne y_j$, converging to $\,z$. Then
\begin{enumerate}
  \def\theenumi{{\rm\alph{enumi}}}
\item[{\rm(i)}] $L\subseteq\tzn$ if $\,N$ is a sub\-man\-i\-fold of 
$\,M\hs$ and $\,x_j,y_j\in N$ for all $\,j$,
\item[{\rm(ii)}] $L\subseteq\mathrm{Ker}\hskip2.2pt\partial\hh\psi_z\,$ 
whenever $\,\psi(x_j)=\psi(y_j)=0\,$ for all $\,j\,$ and some section 
$\,\psi\,$ of a vector bundle over $\,M\nh$.
\end{enumerate}
In fact, we obtain (ii) by identifying a neighborhood $\,\,U'$ of $\,z\,$ in 
$\,M\,$ with a neighborhood of $\,z=0\,$ in the vector space $\,W\nnh=\tzm\,$ 
and trivializing the bundle over $\,\,U'\nnh$, so that $\,\psi\,$ becomes a 
vec\-tor-val\-ued function $\,(\psi^1\nnh,\dots,\psi^q):U'\nh\to\bbR^q\nnh$. 
Vanishing of each $\,\psi^a$, $\,a=1,\dots,q$, at both points $\,x_j,y_j$ 
implies that $\,\dot x_j(t^a_j)\in\mathrm{Ker}\hskip2.2ptd\hh\psi^a_y$, where 
$\,x_j(t)=x_j\nh+t\hh(y_j\nh-x_j)\,$ and $\,y=x_j(t^a_j)\,$ for each fixed 
$\,a\,$ and some sequence $\,t^a_j\nh\in(0,1)$, $\,j=1,2,\ldots\hs$. The 
convergence relation $\,\bbR(y_j\nh-x_j)\to L\,$ now yields 
$\,L\subseteq\mathrm{Ker}\hskip2.2ptd\hh\psi^a_z$ for each $\,a$, and 
(ii) follows. Now (ii) yields (i), since $\,N\nh=\psi^{-1}(0)\,$ for a 
vec\-tor-val\-ued function $\,\psi\,$ which is a submersion onto a 
neighborhood of $\,0\,$ in a vector space.
\end{remark}
In the following lemma, which will be needed in Section~\ref{cz}, convergence 
of tangent directions refers to the appropriate Gras\-mann\-i\-an bundle, and 
can also be interpreted as convergence in $\,\tm\,$ of suitably normalized 
spanning vectors. For the definitions of sub\-con\-vex\-i\-ty and rigidity, 
see Section~\ref{nc}.
\begin{lemma}\label{sbcnv}Suppose that\/ $\,\,U'$ is a sub\-con\-vex 
nor\-mal-co\-or\-di\-nate neighborhood of a point\/ $\,z\,$ in a manifold\/ 
$\,M\,$ with a connection\/ $\,\nabla\nnh$, a line\/ $\,L\,$ through\/ 
$\,0\,$ in\/ $\,\tzm\hs$ is the connecting limit for a pair of sequences 
$\,x_j,y_j\in U'\nnh$, $\,j=1,2,\ldots\hs$, both converging to $\,z$, with\/ 
$\,x_j\nh\ne y_j$ for all\/ $\,j$, and\/ $\,[\hs0,1\hh]\ni t\mapsto x_j(t)\,$ 
denotes the rigid\/ $\,\nabla\nnh$-ge\-o\-des\-ic segment joining\/ 
$\,x_j$ to\/ $\,y_j$ in\/ $\,\,U'\nnh$. Then\/ $\,x_j(t)\to z\,$ and\/ 
$\,\bbR\dot x_j(t)\to L\,$ as\/ $\,j\to\infty$, uniformly in\/ 
$\,t\in[\hs0,1\hh]$.
\end{lemma} 
\begin{proof}For $\,\varOmega\,$ associated with $\,\,U'$ as in 
Section~\ref{nc}, $\,(x,w)\in\varOmega$, and any $\,t\in[\hs0,1\hh]$, let us 
set $\,x(t)=\mathrm{exp}\hh_x\hskip1.3pttw$. Then 
$\,(x,x(t))=\hh\mathrm{Exp}\hh(x,tw)$, cf.\ Section~\ref{nc}, and so the 
pre\-image of $\,(0,\dot x(t))\in T_{(x,x(t))}(M\times M)\,$ under the 
differential of $\,\hh\mathrm{Exp}\,$ at $\,(x,tw)\in\tm\nh$, is, obviously, 
the vector $\,(x,w)\in\{x\}\times\txm=T_{(x,tw)}(\{x\}\times\txm)
\subseteq T_{(x,tw)}(\tm)$, {\it independent\/} (under this identification) of 
$\,t$.

Passing to a subsequence, if necessary, we may assume that 
$\,\bbR\dot x_j(0)\to L\nnh'$ as $\,j\to\infty\,$ for some line 
$\,L\nnh'\nh\subseteq\tzn\nh$. Since 
$\,(x_j(0),x_j(t))=\hh\mathrm{Exp}\hh(x_j(0),tw_j)\,$ for 
$\,w_j\nh=\dot x_j(0)$, the relations $\,x_j(t)\to z$, where $\,t=0,1$, and 
$\,\bbR\dot x_j(0)\to L\nnh'$ amount to $\,(x_j(0),w_j)\to(z,0)\,$ and 
$\,(x_j(0),c_jw_j)\to(z,u)\,$ in $\,\tm\,$ with suitable 
$\,c_j\nh\in(0,\infty)\,$ and a vector $\,u\in L\nnh'\smallsetminus\{0\}$. The 
former relation clearly gives $\,(x_j(0),tw_j)\to(z,0)\in\tm\nh$, and the 
latter, combined with the remark about independence of $\,t\,$ made in the 
last paragraph, implies that $\,\bbR\dot x_j(t)\to L\nnh'\nnh$. In both cases, 
the convergence is uniform in $\,t$.

Let us now identify $\,\,U'$ with a neighborhood of $\,0\,$ in a vector space 
$\,W\nnh$, which also trivializes $\,\tm\,$ over $\,\,U'\nnh$, and fix a norm 
$\,|\hskip3pt|\,$ in $\,W\nnh$. Omitting the subscript $\,j\,$ in $\,x_j,y_j$ 
and $\,w_j\nh=\dot x_j(0)$, we thus have $\,y=\mathrm{exp}\hh_x\hs w\,$ and 
$\,w/|w|\to u\,$ for some vector $\,u\,$ spanning $\,L\nnh'\nnh$. From 
(\ref{fwe}) for $\,\varPsi=\mathrm{exp}\hh_x$ we now obtain 
$\,(y-x)/|w|=[\hh\varPsi(w)-\varPsi(0)]/|w|\to u\,$ (cf.\ the lines preceding 
Remark~\ref{submf}), and, again, convergence of $\,(y-x)/|w|\,$ to the 
$\,|\hskip3pt|$-unit vector $\,u\,$ implies the same for the normalized 
sequence $\,(y-x)/|\hh y-x|$. Thus, $\,L\nnh'$ is the connecting limit for 
the pair $\,x_j,y_j$, so that $\,L\nnh'=L$. Since this happens for the 
limit $\,L\nnh'$ of any convergent subsequence of $\,\bbR\dot x_j(0)$, our 
assertion follows.
\end{proof}
Let $\,z\,$ be a point in a manifold $\,M\nh$. If $\,X,Y\nh\subseteq M\nh$, we 
define $\,\bbL_z(X,Y)\,$ to be the set of all connecting limits for pairs 
$\,x_j,y_j$ of sequences in $\,M\,$ such that $\,x_j,y_j$ both converge to 
$\,z$, while $\,x_j\in X$, $\,y_j\in Y$ and $\,x_j\nh\ne y_j$ for all $\,j$. 
By {\it radial limit directions\/} of a subset $\,Z\subseteq M\,$ at a point 
$\,z\in M\,$ we mean elements of $\,\bbL_z(\{z\},Z)$. Radial limit directions 
of a sub\-man\-i\-fold $\,N\subseteq M\,$ at a point $\,z\in N\hs$ are the 
same as lines through $\,0\,$ in $\,\tzn\nh$, as one sees choosing the 
dif\-feo\-mor\-phism $\,\varPsi\,$ used to define connecting limits in such a 
way that it maps a neighborhood of $\,0\,$ in $\,\tzn\nh\subseteq\tzm\,$ into 
$\,N\nnh$.

\section{Quadratic forms}\label{qf}
In this section all {\it vector spaces\/} are 
fi\-\hbox{nite\hh-}\hskip0ptdi\-men\-sion\-al and real. Given a symmetric 
bi\-lin\-e\-ar form $\,\langle\,,\rangle\,$ in a vector space $\,W\nnh$, we 
will denote by $\,C=\{x\in W:\langle x,x\rangle=0\}\,$ its null cone, and 
by $\,V^\perp\nh
=\{x\in W:\langle x,\,\cdot\,\rangle=0\hskip5pt\mathrm{on}\hskip4ptV\}\,$ the 
$\,\langle\,,\rangle$-or\-thog\-o\-nal complement of a vector subspace 
$\,V\subseteq W\nnh$. Thus, $\,W^\perp$ is the null\-space of 
$\,\langle\,,\rangle$. The {\it quadratic function\/} $\,Q:W\to\bbR\,$ 
corresponding to $\,\langle\,,\rangle\,$ is given by 
$\,Q(x)=\langle x,x\rangle$, and so its differential at any 
$\,x\in W\hs$ is $\,d\hh Q_x\nh=2\langle x,\,\cdot\,\rangle$. Consequently,
\begin{equation}\label{wpc}
\mathrm{the\ null\-space\ }\,\hs W^\perp\hs\mathrm{\ coincides\ with\ 
the\ set\ of\ critical\ points\ of\ }\hs\,Q\hs.
\end{equation}
\begin{remark}\label{semdf}Let a symmetric bi\-lin\-e\-ar form 
$\,\langle\,,\rangle\,$ in a vector space $\,W\hs$ be sem\-i\-def\-i\-nite. 
Then $\,\langle\,,\rangle\,$ satisfies the Schwarz inequality 
$\,\langle x,y\rangle^2\le\hs\langle x,x\rangle\langle y,y\rangle\,$ for 
$\,x,y\in W\nnh$. (In fact, changing the sign of $\,\langle\,,\rangle\,$ if 
necessary so as to make it positive sem\-i\-def\-i\-nite, we can 
approximate it with pos\-i\-tive-def\-i\-nite forms 
$\,\langle\,,\rangle+\ve\langle\,,\rangle\nnh_+^{\phantom i}$, where 
$\,\ve>0\,$ and $\,\langle\,,\rangle\nnh_+^{\phantom i}$ is positive 
definite.) Consequently, its null cone $\,C\,$ coincides with its 
null\-space $\,W^\perp\nnh$, and so $\,C\,$ is a vector subspace of $\,W\nnh$. 
Thus, $\,C\,$ is a (sin\-gu\-lar\-i\-ty\nh-free) sub\-man\-i\-fold of 
co\-di\-men\-sion $\,\mathrm{rank}\hskip1.7pt\langle\,,\rangle\,$ in $\,W\nnh$.
\end{remark}
\begin{remark}\label{snglr}Let $\,C\,$ be the null cone and $\,W^\perp$ the 
null\-space of a symmetric bi\-lin\-e\-ar form $\,\langle\,,\rangle\,$ in a 
vector space $\,W\hs$ which is not sem\-i\-def\-i\-nite. Then
\begin{enumerate}
  \def\theenumi{{\rm\alph{enumi}}}
\item[{\rm(a)}] the set of singular points of $\,C\,$ is nonempty, and 
coincides with $\,W^\perp\nnh$,
\item[{\rm(b)}] the nonsingular subset $\,C\smallsetminus W^\perp$ is dense 
in $\,C$,
\item[{\rm(c)}] the connected components of $\,C\smallsetminus W^\perp$ are 
co\-di\-men\-\hbox{sion\hh-}\hskip.7ptone sub\-man\-i\-folds of $\,W\nnh$,
\item[{\rm(d)}] for $\,y\in C$, denoting by $\,E_y\nh\subseteq W\hs$ the 
union of all radial limit directions of $\,C\,$ at $\,y\,$ (defined at the end 
of Section~\ref{cl}), we have $\,E_y\nh=C\,$ if $\,y\in W^\perp\nnh$, and 
$\,E_y\nh=T_yC=y^\perp$ if $\,y\in C\smallsetminus W^\perp\nnh$, so that in 
the former case $\,E_y$ spans $\,W\nnh$, and in the latter $\,E_y$ is a 
co\-di\-men\-\hbox{sion\hh-}\hskip.7ptone sub\-space of $\,W\nnh$.
\end{enumerate}
Namely, we have (c) since $\,0\,$ is a regular value of the function 
$\,u\mapsto\langle u,u\rangle\,$ restricted to 
$\,W\nnh\smallsetminus W^\perp\nnh$ (the differential of which, at any 
$\,u\in W\nnh$, is $\,2\langle u,\,\cdot\,\rangle$). Also, 
$\,\langle\,,\rangle\,$ descends to a symmetric bi\-lin\-e\-ar form in 
$\,W/W^\perp$ which is nondegenerate and indefinite, so that it has nonzero 
null vectors lying arbitrarily close to $\,0$, and (b) follows.

Next, $\,C\,$ spans $\,W\nnh$. In fact, we may choose a 
$\,\langle\,,\rangle$-or\-thog\-o\-nal basis $\,w_j,u_a,v_\mu$, where the 
index $\,j\,$ (or $\,a$, or $\,\mu$) ranges between $\,1\,$ and some 
$\,i_+^{\phantom i}\ge 1\,$ (or, some $\,i_-^{\phantom i}\ge 1$, or, 
respectively, some $\,k\ge0$), while 
$\,\langle w_j,w_j\rangle=1=-\hh\langle u_a,u_a\rangle\,$ and 
$\,\langle v_\mu,v_\mu\rangle=0$. Thus, $\,W\hs$ has a basis of 
$\,\langle\,,\rangle$-null vectors, formed by all 
$\,w_1^{\phantom i}-u_a$, all $\,u_1^{\phantom i}+w_j$, and all $\,v_\mu$. Now 
(d) is immediate from (c) and the final sentence of Section~\ref{cl}, while 
(a) is an obvious consequence of (d). 
\end{remark}
\begin{remark}\label{rgvlu}If $\,(\hskip2pt,\hskip.2pt)\,$ is a 
nondegenerate symmetric bi\-lin\-e\-ar form on a vector space $\,W\hs$ and 
$\,\varSigma=\{u\in W:|u|=1\}\,$ denotes the unit sphere of a fixed Euclidean 
norm $\,|\hskip3pt|\,$ in $\,W\nnh$, then $\,0\,$ is a regular value of the 
function $\,\varSigma\ni u\mapsto(u,u)$.

More precisely, the differential of this function at any $\,u\in\varSigma\,$ 
is $\,2(u,\,\cdot\,)\,$ restricted to $\,T_u\varSigma$, which is nonzero when 
$\,(u,u)=0$, or else $\,u\,$ would be 
$\,(\hskip2pt,\hskip.2pt)$-or\-thog\-o\-nal to $\,u\,$ as well as to 
$\,T_u\varSigma$, and hence to the whole space 
$\,W\nnh=\hs\bbR u\hh\oplus T_u\varSigma$.
\end{remark}

\section{Some consequences of the \hbox{Morse\hs-}\hskip0ptBott 
lemma}\label{mb}
The following result is often referred to as the 
{\it\hbox{Morse\hs-}\hskip0ptBott lemma}.
\begin{lemma}\label{mrsbt}Suppose that a sub\-man\-i\-fold $\,K$ of a 
manifold\/ $\,N$ consists of critical points of a function\/ $\,Q:N\to\bbR$, 
while\/ $\,z\in K\,$ and\/ $\,Q(z)=0$. If for the Hess\-ian\/ 
$\,\partial\hs d\hh Q_z$ we have\/ 
$\,\mathrm{rank}\hskip2.9pt\partial\hs d\hh Q_z\nh\ge\dim N\nh-\hh\dim K$, 
then there exists a dif\-feo\-mor\-phism\/ $\,\varPsi\,$ of a neighborhood\/ 
$\,\,U\hs$ of\/ $\,0\,$ in\/ $\,\tzn$ onto a neighborhood\/ $\,\,U'$ of\/ 
$\,z\,$ in\/ $\,M\,$ such that\/ $\,\varPsi(0)=z\,$ and\/ 
$\,d\hs\varPsi_0^{\phantom i}=\hh\mathrm{Id}:\tzn\to\tzn\nnh$, while\/ 
$\,Q\nh\circ\varPsi\,$ equals the restriction to\/ $\,\,U$ of the quadratic 
function of\/ $\,\partial\hs d\hh Q_z$, and\/ 
$\,K\nh\cap\hh U'\nnh=\varPsi(V\nnh\cap\hh U)$, where $\,V\nh\subseteq\tzm\,$ 
is the null\-space of\/ $\,\partial\hs d\hh Q_z$.
\end{lemma}
\begin{proof}See \cite{banyaga-hurtubise}. Note that, as the null\-space of 
$\,\partial\hs d\hh Q_x$ contains $\,\txk\,$ whenever $\,x\in K$, the 
inequality assumed about $\,\mathrm{rank}\hskip2.9pt\partial\hs d\hh Q_x$ at 
$\,x=z\,$ is actually an equality, not just at $\,z$, but also at all nearby 
$\,x\in K$. Also, the requirement that 
$\,d\hs\varPsi_0^{\phantom i}=\hh\mathrm{Id}$, not explicitly mentioned in 
\cite{banyaga-hurtubise}, can easily be realized, as it satisfied when $\,f\,$ 
is already dif\-feo\-mor\-phi\-cal\-ly identified with a quadratic function.
\end{proof}
Consider now a subset $\,Z\,$ of a manifold $\,N\nnh$, a point $\,z\in Z\,$ 
and a symmetric bi\-lin\-e\-ar form $\,(\hskip2pt,\hskip.2pt)\,$ in 
$\,\tzn\nnh$. We will call $\,Z\,$ a {\it quadric at\/} $\,z\,$ {\it in\/} 
$\,N\hs$ {\it modelled on\/} $\,(\hskip2pt,\hskip.2pt)\,$ if there 
exists a dif\-feo\-mor\-phism $\,\varPsi\,$ of a neighborhood $\,\,U\hs$ of 
$\,0\,$ in $\,\tzn$ onto a neighborhood $\,\,U'$ of $\,z\,$ in $\,M\,$ such 
that $\,\varPsi(0)=z\,$ and 
$\,d\hs\varPsi_0^{\phantom i}=\hh\mathrm{Id}:\tzn\to\tzn\nnh$, a s well as 
$\,Z\cap U'\nnh=\varPsi(C\cap U)$, where $\,C=\{u\in\tzn:(u,u)=0\}\,$ is the 
null cone of $\,(\hskip2pt,\hskip.2pt)$.

We may now rephrase one immediate consequence of Lemma~\ref{mrsbt} as follows.
\begin{lemma}\label{quadr}Under the hypotheses of Lemma\/~{\rm\ref{mrsbt}}, 
the zero set\/ $\,Z=Q^{-1}(0)\,$ is a quadric at\/ $\,z\,$ in\/ $\,N$ 
modelled on\/ $\,\partial\hs d\hh Q_z$.
\end{lemma}
\begin{remark}\label{dcomp}A fi\-\hbox{nite\hh-}\hskip0ptdi\-men\-sion\-al 
real vector space with a fixed symmetric bi\-lin\-e\-ar form 
$\,(\hskip2pt,\hskip.2pt)\,$ can always be decomposed into a 
$\,(\hskip2pt,\hskip.2pt)$-or\-thog\-o\-nal direct sum $\,W\nnh\oplus V\hs$ 
of sub\-spaces such that $\,(\hskip2pt,\hskip.2pt)\,$ is nondegenerate on 
$\,W\,$ and $\,(\hskip2pt,\hskip.2pt)=0\,$ on $\,V\nnh$. Consequently, 
$\,V\hs$ is the null\-space of $\,(\hskip2pt,\hskip.2pt)$. Denoting by 
$\,\varSigma\,$ the $\,|\hskip3pt|$-unit sphere around $\,0\,$ in $\,W\nnh$, 
for a fixed Euclidean norm $\,|\hskip3pt|\,$ in $\,W\nnh\oplus V\nnh$, and by 
$\,Q\,$ the quadratic function of $\,(\hskip2pt,\hskip.2pt)$, we clearly have 
$\,Q(su+x)=s^2(u,u)\,$ whenever $\,(s,u,x)\in\bbR\times\varSigma\times V\nnh$. 
Every neighborhood of $\,0\,$ in $\,W\nnh\oplus V\hs$ contains a smaller 
neighborhood of the form $\,B\oplus K=\{y+x:y\in B,\,\,x\in K\}$, where 
$\,B\subseteq W$ is the open $\,|\hskip3pt|$-ball in $\,W$ of some radius 
$\,\ve>0$, centered at $\,0$, and $\,K\,$ is a neighborhood of $\,0\,$ in 
$\,V\nnh$.
\end{remark}
The next lemma will be used in the proof of Theorem~\ref{tgdir}. 
\begin{lemma}\label{ncleq}If\/ $\,Y\nnh,Y'$ are quadrics at a point\/ $\,z\,$ 
in a manifold\/ $\,P\nnh$, both modelled on the same symmetric bi\-lin\-e\-ar 
form\/ $\,(\hskip2pt,\hskip.2pt)\,$ in\/ $\,\tzp\nnh$, and\/ 
$\,Y\nnh\subseteq Y'\nnh$, then\/ $\,U'\nnh\cap Y=\,U'\nnh\cap Y'$ for some 
neighborhood\/ $\,\,U'$ of\/ $\,z\,$ in\/ $\,P\nnh$.
\end{lemma}
\begin{proof}In the case where $\,(\hskip2pt,\hskip.2pt)\,$ is 
sem\-i\-def\-i\-nite, $\,Y$ and $\,Y'$ are sub\-man\-i\-folds of 
co\-di\-men\-sion $\,\mathrm{rank}\hskip1.7pt(\hskip2pt,\hskip.2pt)\,$ in 
$\,P\hs$ (see Remark~\ref{semdf}), and our claim follows from the inverse 
mapping theorem applied to the inclusion $\,Y\nh\to Y'\nnh$.

Suppose now that $\,(\hskip2pt,\hskip.2pt)\,$ is not sem\-i\-def\-i\-nite. 
Using the notations and identifications introduced in Remark~\ref{dcomp}, we 
think of $\,P\hs$ as a neighborhood of $\,0\,$ in 
$\,\tzp\nnh=W\nnh\oplus V\hs$ having the form $\,B\oplus K$, with $\,Y$ and 
$\,Y'$ equal to the zero sets of the quadratic function $\,Q\,$ of 
$\,(\hskip2pt,\hskip.2pt)\,$ and, respectively, of the function $\,Q\hh'$ 
obtained as the composite of $\,Q\,$ with a dif\-feo\-mor\-phism between 
$\,B\oplus K\,$ and a neighborhood of $\,0\,$ in $\,W\nnh\oplus V\nnh$, whose 
value and differential at $\,0\,$ are $\,0\,$ and $\,\mathrm{Id}$. The 
Hess\-i\-ans of $\,Q\,$ and of $\,Q\hh'$ thus both equal 
$\,2(\hskip2pt,\hskip.2pt)$, while, by (\ref{wpc}), the neighborhood $\,K\,$ 
of $\,0\,$ in $\,V\hs$ appearing in the equality $\,P\nh=B\oplus K\,$ is 
precisely the set of critical points of $\,Q$, that is, singular points of 
$\,Y$ (see Remark~\ref{snglr}(a)). In view of the characterization of singular 
and nonsingular points of a quadric, given in the final clause of 
Remark~\ref{snglr}(d), all singular points of $\,Y$ are also singular in 
$\,Y'\nnh$, so that $\,K\,$ must consist of critical points, as well as zeros, 
of $\,Q\hh'\nnh$.  

For the open set $\,\varOmega=(-\hh\ve,\ve)\times\varSigma\times K\,$ in 
$\,\bbR\times\varSigma\times V\nnh$, where $\,\ve\,$ is the radius of 
$\,B\,$ (cf.\ Remark~\ref{dcomp}) and the function $\,\beta:\varOmega\to\bbR\,$ 
given by $\,\beta(s,u,x)=Q\hh'\nh(su+x)$, we thus have 
$\,\beta(0,u,x)=\beta_s(0,u,x)=0\,$ whenever $\,(0,u,x)\in\varOmega\,$ 
(notation of (\ref{bei})), and so
\begin{equation}\label{qsu}
\mathrm{i)}\hskip8ptQ\hh'\nh(su+x)\,=\,s^2\mu(s,u,x)\hs,\hskip18pt\mathrm{ii)}
\hskip8pt\mu(0,u,0)\,=\,(u,u)
\end{equation}
for $\,\mu:\varOmega\to\bbR\,$ with 
$\,\mu(s,u,x)=\int_0^1(1-t)\beta_{ss}^{\phantom i}(ts,u,x)\,dt\,$ and all 
$\,(s,u,x)\in\varOmega$. In fact, (\ref{bei}.ii) yields (\ref{qsu}.i), and 
(\ref{qsu}.ii) follows as 
$\,2\hh\mu(0,u,0)=\hh\partial\hs d\hh Q\hh_{\nh0}'(u,u)$, cf.\ (\ref{hss}).

According to Remark~\ref{rgvlu}, the points $\,u\in\varSigma\,$ such that 
$\,(u,u)=0\,$ form a (possibly disconnected) 
co\-di\-men\-\hbox{sion\hh-}\hskip0ptone sub\-man\-i\-fold $\,\varPi\,$ of 
$\,\varSigma$, and the function $\,\varSigma\ni u\mapsto(u,u)\,$ has a nonzero 
differential at every point of $\,\varPi$. For 
$\,\varOmega_0^{\phantom i}=(-\hh\ve,\ve)\times\varPi\times K$, (\ref{qsu}.ii) 
and compactness of $\,\varPi\,$ allow us to choose $\,\ve\,$ and $\,K\,$ small 
enough so as to ensure that $\,\mu:\varOmega\to\bbR\,$ has a nonzero 
differential at every point of the co\-di\-men\-\hbox{sion\hh-}\hskip0ptone 
sub\-man\-i\-fold $\,\varOmega_0^{\phantom i}$ of $\,\varOmega$. Also, by 
(\ref{qsu}.i), $\,\mu=0\,$ on $\,\varOmega_0^{\phantom i}$, since the equality 
$\,Q(su+x)=s^2(u,u)\,$ (see Remark~\ref{dcomp}) gives 
$\,\varOmega_0^{\phantom i}\subseteq Y\nnh\subseteq Y'\nnh$. 
Lemma~\ref{nzero}(b) now guarantees, for smaller $\,\ve\,$ and $\,K$, the 
existence of an open subset $\,\,U\,$ of $\,\varSigma\,$ such that 
$\,\varPi\subseteq U\,$ and $\,\mu\ne0\,$ everywhere in 
$\,(-\hh\ve,\ve)\times[\hh U\smallsetminus\varPi\hh]\times K$. Since, in 
addition, (\ref{qsu}.ii) yields $\,\mu\ne0\,$ at all points of the compact set 
$\,\{0\}\times[\varSigma\smallsetminus U\hh]\times\{0\}$, making 
$\,\ve\,$ and $\,K\,$ even smaller we obtain $\,\mu\ne0\,$ everywhere in 
$\,(-\hh\ve,\ve)\times[\varSigma\smallsetminus\varPi\hh]\times K
=\varOmega\smallsetminus\varOmega_0^{\phantom i}$. In view of (\ref{qsu}.i), 
this proves the lemma, with $\,\,U'\nh=B\oplus K\,$ for the current choices of 
$\,\ve\,$ and $\,K$.
\end{proof}
The following result is a key technical ingredient for the proof of 
Theorem~\ref{nbzrs}. For the definition of 
$\,\bbL_z(Z\smallsetminus\phi^{-1}(0),K)$, see Section~\ref{cl}.
\begin{theorem}\label{tgdir}Given a sub\-man\-i\-fold\/ $\,K$ of a manifold\/ 
$\,N\nh$, a point\/ $\,z\in K$, a vector space\/ $\,\mathcal{T}\hs$ containing 
$\,\tzn$ as a subspace, a symmetric bi\-lin\-e\-ar form\/ 
$\,\langle\,,\rangle\,$ in\/ $\,\mathcal{T}\nnh$, a vec\-tor-val\-ued 
function\/ $\,f:N\to\mathcal{T}\nnh$, and a function\/ $\,\phi:N\to\bbR$, for 
which\/ $\,d\phi_z\nh\ne0\,$ and\/ $\,P\nh=\phi^{-1}(0)\,$ is a 
co\-di\-men\-\hbox{sion\hh-}\hskip0ptone sub\-man\-i\-fold of\/ $\,N$ such 
that\/ $\,K\subseteq Y\nnh\subseteq P$ with some quadric\/ $\,Y$ at\/ $\,z\,$ 
in the manifold\/ $\,P$ modelled on the restriction of\/ 
$\,\langle\,,\rangle\,$ to the subspace\/ 
$\,H=\tzp\nnh$, let\/ $\,\partial\hs d\hskip-.8ptf_{\nh z}$ denote the 
Hess\-i\-an\/ of\/ $\,f\hs$ at\/ $\,z$. In addition, suppose that
\begin{enumerate}
  \def\theenumi{{\rm\alph{enumi}}}
\item[{\rm(a)}] the restriction of\/ $\,\langle\,,\rangle\,$ to\/ $\,\tzn$ is 
nonzero,
\item[{\rm(b)}] $V\nnh=\hs\tzk\,$ is the null\-space of the restriction of\/ 
$\,\langle\,,\rangle\,$ to\/ $\,H=\tzp\nnh$,
\item[{\rm(c)}] $d\hskip-.8ptf=0\,$ at all points of\/ $\,K$, and\/ 
$\,Y\nnh\subseteq Z$, where\/ $\,Z\subseteq N$ is the zero set of\/ $\,f$,
\item[{\rm(d)}] $\langle w,\partial\hs d\hskip-.8ptf_{\nh z}\rangle
=\hs d\hh\phi_z\nh\otimes\langle w,\,\cdot\,\rangle
+\langle w,\,\cdot\,\rangle\otimes\hs d\hh\phi_z\nh
-[\hh d\hh\phi_z(w)]\hh\langle\,,\rangle\,$ for every\/ $\,w\in\tzn\nh$.
\end{enumerate}
Then
\begin{enumerate}
  \def\theenumi{{\rm\roman{enumi}}}
\item[{\rm(i)}] $Z\cap P\cap\varOmega\,\subseteq\,Y$ for some neighborhood\/ 
$\,\varOmega\,$ of\/ $\,z\,$ in\/ $\,N\nnh$,
\item[{\rm(ii)}] no element of\/ $\,\bbL_z(Z\smallsetminus\phi^{-1}(0),K)\,$ 
is contained in\/ $\,(H\smallsetminus V\hh)\cup\{0\}$.
\end{enumerate}
\end{theorem}

\section{Proof of Theorem~\ref{tgdir}}\label{pm}
In view of (a), both $\,\langle\,,\rangle\,$ and $\,d\phi_z$ are nonzero on 
$\,\tzn\nnh$. Let us fix $\,w\in\tzn\hs$ such that 
$\,d\phi_z(w)\ne0\ne\langle w,w\rangle$. 
The Hess\-i\-an $\,\partial\hs d\hh Q_z$ of the function $\,Q:N\to\bbR\,$ with 
$\,Q(y)=\langle w,f(y)\rangle\,$ obviously equals the right-hand side in (d). 
Thus, if we set $\,\xi=d\hh\phi_z$,
\begin{equation}\label{qxe}
\begin{array}{l}
\partial\hs d\hh Q_z(w,w)=[\hs\xi(w)]\hh\langle w,w\rangle\ne0\hskip11pt
\mathrm{and,\ for}\hskip8ptu\in H=\tzp\nnh=\hh\mathrm{Ker}\hskip2.2pt\xi\hs,\\
\partial\hs d\hh Q_z(u,u)=-\hs[\hs\xi(w)]\hh\langle u,u\rangle\hs,
\hskip25pt\partial\hs d\hh Q_z(w,u)=0\hs.
\end{array}
\end{equation}
\subsection{{\bf Assertion} {\rm (i)}\label{ai}}
By (b) and (\ref{qxe}), the assumptions of Lemma~\ref{mrsbt} hold for 
$\,P\hs$ (rather than $\,N$), our $\,z$, the restriction of $\,Q\,$ to 
$\,P\nh$, and $\,K$. Therefore, in view of Lemma~\ref{quadr}, 
$\,Y'\nnh=P\cap Q^{-1}(0)\,$ is a quadric at $\,z\,$ in $\,P\hs$ modelled 
on the restriction to $\,\tzp\hs$ of $\,\partial\hs d\hh Q_z$ or, 
equivalently, of 
$\,\langle\,,\rangle\,$ (cf.\ (\ref{qxe})). Lemma~\ref{ncleq} thus applies 
to $\,Y'$ and the quadric $\,Y\hs$ in the statement of Theorem~\ref{tgdir}, 
as the hypotheses $\,Y\nnh\subseteq P\hs$ and $\,Y\nnh\subseteq Z\,$ in 
Theorem~\ref{tgdir}, combined with the relation $\,Z\subseteq Q^{-1}(0)\,$ 
(obvious from the definitions of $\,Z\,$ and $\,Q$), give 
$\,Y\nnh\subseteq Z\cap P\subseteq Y'\nnh$. In view of Lemma~\ref{ncleq}, the 
latter inclusions turn into equalities if one replaces the sets involved by 
their intersections with a suitable neighborhood $\,\varOmega\,$ of $\,z\,$ 
in $\,N\nnh$. This not only yields the conclusion 
$\,Z\cap P\cap\varOmega\,\subseteq\,Y\hs$ claimed in (i), but also shows 
that
\begin{equation}\label{vez}
f=0\,\,\mathrm{\ at\ all\ points\ of\ }\,\,P\cap Q^{-1}(0)\,\,\mathrm{\ 
sufficiently\ close\ to\ }\,\,z\hs.
\end{equation}
The remainder of this section is devoted to proving assertion (ii).

\subsection{{\bf Identifications and decompositions}\label{id}}
In view of (b) and (\ref{qxe}), the hypotheses of Lemma~\ref{mrsbt} are also 
satisfied by our $\,N\nnh,z,Q\,$ and $\,K$. Replacing $\,N\hs$ by a 
neighborhood of $\,z\,$ in $\,N\nnh$, we may thus use Lemma~\ref{mrsbt} to 
identify $\,N\hs$ with a neighborhood of $\,0\,$ in the vector space 
$\,\tzn\nnh$, and $\,z\,$ with $\,0$, in such a way that $\,Q\,$ becomes the 
quadratic function of the symmetric bi\-lin\-e\-ar form 
$\,(\hskip2pt,\hskip.2pt)=\partial\hs d\hh Q_z$ on $\,\tzn\nnh$, and $\,K\,$ 
is the intersection of $\,N\hs$ with the null\-space of 
$\,(\hskip2pt,\hskip.2pt)$. We also decompose $\,\tzn\hs$ into a direct sum 
$\,W\nnh\oplus V\hs$ as in Remark~\ref{dcomp}, choosing $\,W$ so that 
$\,w\in W\nnh$. Thus, $\,V\hs$ is the null\-space of 
$\,(\hskip2pt,\hskip.2pt)$. As a result, we obtain three 
$\,(\hskip2pt,\hskip.2pt)$-or\-thog\-o\-nal decompositions:
\begin{equation}\label{dec}
\tzn=\,W\nnh\oplus V,\hskip16ptW=\,\bbR w\oplus H',\hskip10pt
H\hs=\,H'\nnh\oplus V,\hskip16pt
\end{equation}
where $\,H'\nnh=w^\perp\nnh\cap W\nnh$, with $\,(\hskip3pt)^\perp$ standing 
for the $\,(\hskip2pt,\hskip.2pt)$-or\-thog\-o\-nal complement in 
$\,\tzn\nnh$.

As $\,V\hs$ is the null\-space of $\,(\hskip2pt,\hskip.2pt)$, 
we have $\,K\nh=V\nnh\cap N\nnh$. Replacing $\,N\hs$ and $\,K\,$ with smaller 
neighborhoods of $\,0\,$ in $\,\tzn\hs$ and $\,V\nnh$, we thus get 
$\,N\nnh=B\oplus K$, meaning that $\,N\nnh=\{y+x:y\in B,\,\,x\in K\}$, where 
$\,B\subseteq W$ is the open $\,|\hskip3pt|$-ball in $\,W$ of some radius 
$\,\ve>0$, centered at $\,0$, for a fixed Euclidean norm $\,|\hskip3pt|\,$ in 
$\,\tzn\nnh$. Summarizing, we have
\begin{equation}\label{sum}
H'\nh=H\cap W,\hskip22ptH'\hs\subseteq\,H=\tzp\nh=\mathrm{Ker}\hskip2.2pt\xi
=w^\perp,\hskip22ptV\nh=\,\tzk\hh.
\end{equation}
As $\,(\hskip2pt,\hskip.2pt)=\partial\hs d\hh Q_z$ satisfies (\ref{qxe}), it 
follows from (b) and (\ref{dec}) that
\begin{equation}\label{rst}
\begin{array}{l}
\mathrm{a)}\hskip6pt\mathrm{the\ restriction\ of\ }\,(\hskip2pt,\hskip.2pt)\,
\mathrm{\ to\ }\,H'\hs\mathrm{\ is\ nondegenerate,}\\
\mathrm{b)}\hskip6pt\mathrm{if\ }\,(\hskip2pt,\hskip.2pt)\,\mathrm{\ is\ 
positive\ or\ negative\ definite\ on\ }\,H',\hs\mathrm{\ so\ must\ be\ 
}\,\langle\,,\rangle\hs,
\end{array}
\end{equation}
(\ref{rst}.b) being obvious since $\,(\hskip2pt,\hskip.2pt)\,$ restricted to 
$\,H'$ is, by (\ref{qxe}), a nonzero multiple of $\,\langle\,,\rangle$.

We use the symbol $\,\varSigma\,$ for the $\,|\hskip3pt|$-unit sphere around 
$\,0\,$ in $\,W\nnh$.

\subsection{{\bf Factorizations of\/ $\,F\nnh,\phi,Q$, and a description of 
$\,\bbL_z(Z\smallsetminus\phi^{-1}(0),K)$}\label{fd}}
From now on $\,(s,u,x)\,$ denotes a generic element of the open set 
$\,\varOmega=(-\hh\ve,\ve)\times\varSigma\times K\,$ in 
$\,\bbR\times\varSigma\times V\nnh$. We define a $\,C^\infty$ function 
$\,\beta:\varOmega\to\tzn\hs$ by $\,\beta(s,u,x)=f(su+x)$. As $\,f\,$ 
and $\,d\hskip-.8ptf\,$ vanish on 
$\,V\nnh\cap N\nnh=K\subseteq Z$, we have $\,\beta(0,u,x)=f(0,x)=0\,$ as well 
as $\,\beta_s(0,u,x)=d\hskip-.8ptf_{\nh x}(u)=0\,$ whenever 
$\,(0,u,x)\in\varOmega\,$ (notation of (\ref{bei})). Similarly, the function 
$\,\gamma(s,u,x)=\phi(su+x)\,$ vanishes when $\,s=0$. Thus, 
$\,\beta(s,u,x)\,$ is smoothly divisible by $\,s^2\nnh$, and 
$\,\gamma(s,u,x)\,$ by $\,s$. Explicitly, according to (\ref{bei}),
\begin{equation}\label{vsu}
\begin{array}{l}
f(su+x)=s^2F(s,u,x)\hskip8.38pt\mathrm{with}\hskip4.4pt
F(s,u,x)=\int_0^1(1-t)\beta_{ss}^{\phantom i}(ts,u,x)\,dt\hs,\\
\phi(su+x)=s\hs\varPhi(s,u,x)\hs,\hskip12pt\mathrm{where}\hskip7pt
\varPhi(s,u,x)=\int_0^1\gamma_s(ts,u,x)\,dt\hs,\\
Q(su+x)\,=\,s^2(u,u)\,=\,s^2\langle w,F(s,u,x)\rangle\hs.
\end{array}
\end{equation}
By (\ref{vsu}), a vector spanning a line 
$\,L\in\bbL_z(Z\smallsetminus\phi^{-1}(0),K)\,$ is, up to a factor, the limit 
of a sequence $\,v_j/|v_j|$, $\,j=1,2,\dots\hs$, where 
$\,v_j\nh=s_ju_j\nh+x_j\nh-y_j$ with $\,(s_j,u_j,x_j)\in\varOmega\,$ and 
$\,y_j\nh\in K\,$ such that $\,(s_j,x_j,y_j)\to(0,0,0)$, as well as 
\begin{equation}\label{fsu}
\varPhi(s_j,u_j,x_j)\,\ne\hs\,0\,=\,F(s_j,u_j,x_j)\,=\,(u_j,u_j)\hskip14pt
\mathrm{for\ all}\hskip8ptj\ge1\hh.
\end{equation}
Passing to a subsequence, we may further assume that $\,u_j\nh\to u\,$ for 
some $\,u\in\varSigma$, while $\,s_ju_j/|v_j|\to c\hh u\,$ and 
$\,(x_j\nh-y_j)/|v_j|\to x\,$ for some $\,c\in\bbR\,$ and $\,x\in V\nnh$. 
Thus, since $\,F\,$ is continuous, $\,F(0,u,0)=(u,u)=0$. Also, 
$\,L=\bbR(c\hh u+x)$, so that, by (\ref{sum}),
\begin{equation}\label{cez}
\mathrm{if}\hskip8pt\xi(u)\ne\hs0\hs,\hskip8pt\mathrm{then}\hskip7pt
L\hskip7pt\mathrm{is\ not\ contained\ 
in}\hskip7pt(H\smallsetminus V\hh)\cup\{0\}\hs.
\end{equation}

\subsection{{\bf Values on} $\,\{0\}\times\varSigma\times\{0\}$, {\bf where} 
$\,\varSigma\,$ {\bf is the $\,|\hskip3pt|$-unit sphere in} $\,W$\label{vo}}
For $\,\xi=d\hh\phi_z$ and any $\,u\in\varSigma$, the definitions of 
$\,\beta\,$ and $\,\gamma\,$ along with (\ref{hss}) and (\ref{vsu}) give
\begin{equation}\label{azu}
\mathrm{i)}\hskip8pt2\hh F(0,u,0)\,
=\,\hh\partial\hs d\hskip-.8ptf_{\nh z}(u,u)\hs,\hskip18pt\mathrm{ii)}\hskip8pt
\varPhi(0,u,0)=\xi(u)\hs.
\end{equation}
Consequently, using (d) we see that, if $\,u\in\varSigma\,$ and 
$\,u\hh'\nh\in T_u\varSigma$, while $\,w\hh'\nh\in\tzn\nnh$,
\begin{equation}\label{wdv}
\begin{array}{l}
\mathrm{a)}\hskip6pt2\hh\langle w\hh'\nnh,F(0,u,0)\rangle
=2\hh[\hs\xi(u)]\hh\langle w\hh'\nnh,u\rangle
-[\hs\xi(w\hh'\hh)]\hh\langle u,u\rangle\hs,\\
\mathrm{b)}\hskip6pt
\langle w\hh'\nnh,d\hskip-.8ptF_{(0,u,0)}^{\phantom i}u\hh'\hh\rangle
=[\hs\xi(u)]\hh\langle w\hh'\nnh,u\hh'\hh\rangle
+[\hs\xi(u\hh'\hh)]\hh\langle w\hh'\nnh,u\rangle
-[\hs\xi(w\hh'\hh)]\hh\langle u,u\hh'\hh\rangle\hs,
\end{array}
\end{equation}
with $\,u\hh'$ on the left-hand side standing for the vector 
$\,(0,u\hh'\nnh,0)\,$ tangent to $\,\{0\}\times\varSigma\times\{0\}\,$ at 
$\,(0,u,0)$.

In the remainder of the proof, $\,(-\hh\ve,\ve)\,$ and $\,K\,$ will 
repeatedly be replaced with smaller neighborhoods of $\,0\,$ in $\,\bbR\,$ 
and $\,V\nnh$, as needed for the argument.

\subsection{{\bf Case A:} $\,(\hskip2pt,\hskip.2pt)\,$ is 
sem\-i\-def\-i\-nite on\/ $\,H$\label{aa}}
By (\ref{rst}.a), $\,(\hskip2pt,\hskip.2pt)\,$ restricted to $\,H'$ is 
positive or negative definite. Furthermore,
\begin{equation}\label{xnz}
\xi(u)\ne0\,\mathrm{\ for\ every\ }\,u\in\varSigma\,\mathrm{\ such\ that\ 
}\,F(0,u,0)=0\hs,
\end{equation}
where $\,\varSigma\subseteq W$ is the $\,|\hskip3pt|$-unit sphere, 
$\,\xi=d\hh\phi_z$ and $\,F\,$ is given by (\ref{vsu}). In fact, suppose that 
$\,u\in\varSigma\,$ and $\,\xi(u)=0$. Since $\,\varSigma\subseteq W$, 
(\ref{sum}) then gives $\,u\in H'\nnh$, and so $\,\langle u,u\rangle\ne0\,$ in 
view of (\ref{rst}.b). Thus, $\,F(0,u,0)\ne0$, as one sees evaluating 
(\ref{wdv}.a) for $\,w\hh'$ equal to the vector $\,w\in W\hs$ with 
$\,\xi(w)\ne0\,$ which appears in (\ref{qxe}).

Assertion (ii) now follows: for any $\,(s_j,u_j,x_j),y_j,c,u\,$ and $\,x\,$ 
with the properties listed in the lines following (\ref{vsu}), including 
$\,F(0,u,0)=0$, (\ref{xnz}) and (\ref{cez}) yield (ii).

\subsection{{\bf Case B:} $\,(\hskip2pt,\hskip.2pt)\,$ is not 
sem\-i\-def\-i\-nite on\/ $\,H$\label{cc}}
This time, $\,(\hskip2pt,\hskip.2pt)\,$ restricted to $\,H'$ is nondegenerate 
and indefinite, cf.\ (\ref{rst}.a) and (\ref{dec}). As before, $\,\varSigma\,$ 
denotes the $\,|\hskip3pt|$-unit sphere in $\,W\nnh$, and $\,\xi=d\hh\phi_z$. 
In view of Remark~\ref{rgvlu}, the condition $\,(u,u)=0\,$ imposed on 
$\,u\in\varSigma\,$ defines a (possibly disconnected) 
co\-di\-men\-\hbox{sion\hh-}\hskip0ptone sub\-man\-i\-fold $\,\varPi\,$ of 
$\,\varSigma$, containing the subset 
$\,\varLambda=\varPi\cap H'\nh=\varPi\cap\hh\mathrm{Ker}\hskip2.2pt\xi\,$ 
(cf.\ (\ref{sum})).

The set $\,\varLambda$, nonempty due to in\-def\-i\-nite\-ness of 
$\,(\hskip2pt,\hskip.2pt)\,$ on $\,H'\nnh$, contains no critical points of the 
restriction $\,\xi:\varPi\to\bbR$. In fact, let $\,u\in\varLambda$, so that 
$\,(u,u)=0\,$ and $\,u\in\varSigma\hh'\nnh$, where 
$\,\varSigma\hh'\nh=\varSigma\cap H'$ is the $\,|\hskip3pt|$-unit sphere in 
$\,H'\nnh$. Remark~\ref{rgvlu} applied to $\,\varSigma\hh'$ shows that the 
functional $\,(u,\,\cdot\,)\,$ is nonzero on 
$\,T_u\varSigma\hh'\nh=T_u\varSigma\cap\hh\mathrm{Ker}\hskip2.2pt\xi\,$ (cf. 
(\ref{sum})), and $\,\xi\,$ is nonzero on $\,T_u\varSigma\,$ (as 
$\,T_u\varSigma\hh'\nh=T_u\varSigma\cap\hh\mathrm{Ker}\hskip2.2pt\xi\,$ is a 
proper subspace of $\,T_u\varSigma$). The restrictions of the 
functionals $\,(u,\,\cdot\,)\,$ and $\,\xi\,$ to $\,T_u\varSigma\,$ are thus 
linearly independent. Consequently, $\,\xi\,$ is nonzero on 
$\,T_u\varSigma\cap\hh\mathrm{Ker}\hskip2.2pt(u,\,\cdot\,)=T_u\varPi$, as 
required.

In terms of 
$\,\varOmega_0^{\phantom i}=(-\hh\ve,\ve)\times\varPi\times K$, these 
conclusions and (\ref{azu}.ii) imply that the nonempty set 
$\,\varLambda'\nh=\{0\}\times\varLambda\times\{0\}\,$ 
consists precisely of all zeros of the function $\,\varPhi:\varOmega\to\bbR\,$ 
given by (\ref{vsu}) which lie in the sub\-man\-i\-fold 
$\,\{0\}\times\varPi\times\{0\}\,$ of $\,\varOmega_0^{\phantom i}$, and that 
the restriction of $\,\varPhi\,$ to $\,\{0\}\times\varPi\times\{0\}\,$ has a 
nonzero differential at every point of $\,\varLambda'\nnh$. 
As $\,\varPi\,$ is compact, choosing smaller $\,\ve\,$ and 
$\,K\,$ we can ensure that $\,0\,$ is a regular value of the restriction 
$\,\varPhi:\varOmega_0^{\phantom i}\to\bbR$. In addition, the sub\-mer\-sion 
$\,s:\varOmega_0^{\phantom i}\to(-\hh\ve,\ve)$, given by 
$\,(s,u,x)\mapsto s$, is constant on $\,\{0\}\times\varPi\times\{0\}$, and so 
$\,\varPhi\,$ and $\,s$, as functions on $\,\varOmega_0^{\phantom i}$, have 
linearly independent differentials at each point of $\,\varLambda'\nnh$. 
For even smaller $\,\ve\,$ and $\,K$, we thus have 
$\,d\hh\varPhi\wedge ds\ne0\,$ everywhere in 
$\,\varOmega_0^{\phantom i}\cap\varPhi^{-1}(0)$. Thus, 
$\,\varOmega_0^{\phantom i}\cap\varPhi^{-1}(0)\,$ is a (possibly disconnected) 
co\-di\-men\-\hbox{sion\hh-}\hskip0ptone sub\-man\-i\-fold of 
$\,\varOmega_0^{\phantom i}$, and the additional condition $\,s=0\,$ defines a 
further co\-di\-men\-\hbox{sion\hh-}\hskip0ptone sub\-man\-i\-fold of 
$\,\varOmega_0^{\phantom i}\cap\varPhi^{-1}(0)\,$ (so that, in particular, 
$\,s\ne0\,$ on a dense subset of 
$\,\varOmega_0^{\phantom i}\cap\varPhi^{-1}(0)$). Next, for $\,F\,$ given by 
(\ref{vsu}),
\begin{enumerate}
  \def\theenumi{{\rm\roman{enumi}}}
\item[{\rm($*$)}] $F:\varOmega_0^{\phantom i}\to\mathcal{T}\hs$ has the 
value $\,0\,$ and nonzero differential at every point $\,(s,u,x)\,$ of the 
co\-di\-men\-\hbox{sion\hh-}\hskip0ptone sub\-man\-i\-fold 
$\,\varOmega_0^{\phantom i}\cap\varPhi^{-1}(0)\,$ of 
$\,\varOmega_0^{\phantom i}$, containing 
$\,\varLambda'\nh=\{0\}\times\varLambda\times\{0\}$.
\end{enumerate}

\subsection{{\bf Justification of} {\rm($*$)}\label{ja}}
As $\,P\nh=\phi^{-1}(0)$, (\ref{vez}) and (\ref{vsu}) along with the 
definitions of $\,\varPi\,$ and $\,\varOmega_0^{\phantom i}$ give $\,F=0\,$ on 
$\,\varOmega_0^{\phantom i}\cap\varPhi^{-1}(0)$, while 
$\,\varLambda'\nh\subseteq\varOmega_0^{\phantom i}\cap\varPhi^{-1}(0)\,$ 
due to the definition 
of $\,\varLambda\,$ and (\ref{azu}.ii). Note that the conclusion here is 
$\,F=0$, rather than just $\,s^2F=0$, since, as mentioned above, the subset 
$\,s\ne0\,$ is dense in $\,\varOmega_0^{\phantom i}\cap\varPhi^{-1}(0)$.

Since $\,\varPi\,$ is compact and we are free to make $\,\ve\,$ and $\,K\,$ 
smaller, ($*$) will follow if we prove it just for $\,(s,u,x)=(0,u,0)\,$ with 
$\,u\in\varLambda$, while restricting $\,F\,$ further, to the 
sub\-man\-i\-fold $\,\{0\}\times\varPi\times\{0\}\,$ of 
$\,\varOmega_0^{\phantom i}$. (As we saw, $\,\varLambda'$ is the intersection 
of $\,\varOmega_0^{\phantom i}\cap\varPhi^{-1}(0)\,$ with 
$\,\{0\}\times\varSigma\times\{0\}$.) We thus only need to show that, whenever 
$\,u\in\varSigma\,$ and $\,(u,u)=\xi(u)=0$, the right-hand side of 
(\ref{wdv}.b) is nonzero for suitable $\,w\hh'\nh\in\tzn\hs$ and 
$\,u\hh'\nh\in W\hs$ with $\,(u,u\hh'\hh)=0$. Let us thus set $\,u\hh'\nh=w$, 
so that $\,(u,u\hh'\hh)=0\hh\ne\hh\xi(u\hh'\hh)\,$ in view of (\ref{qxe}). Also, as 
$\,(\hskip2pt,\hskip.2pt)\,$ is nondegenerate on 
$\,H'$ (see (\ref{rst}.a)), we may choose 
$\,w\hh'\nh\in H'$ with $\,(w\hh'\nnh,u)\ne0\,$ (that is, by (\ref{qxe}), 
$\,\langle w\hh'\nnh,u\rangle\ne0$), and hence $\,\xi(w\hh'\hh)=0\,$ by 
(\ref{sum}). Then 
$\,[\hs\xi(u\hh'\hh)]\hh\langle w\hh'\nnh,u\rangle\ne0$, and the other two 
terms on the right-hand side of (\ref{wdv}.b) vanish.

\subsection{{\bf The final step}\label{fs}}
We now conclude the argument in Case B. First, by ($*$) and 
Lemma~\ref{nzero}(b), $\,F\nh\ne0\,$ everywhere in 
$\,U\smallsetminus\varPhi^{-1}(0)\,$ for some open set 
$\,\,U\smallsetminus\varOmega_0^{\phantom i}$ containing 
$\,\varOmega_0^{\phantom i}\cap\varPhi^{-1}(0)$.

Let us now fix any $\,(s_j,u_j,x_j),y_j,c,u\,$ and $\,x\,$ satisfying the 
conditions in the lines following (\ref{vsu}), which include (\ref{fsu}) 
and $\,F(0,u,0)=(u,u)=0$. We then also have $\,\xi(u)\ne0\,$ (which, in 
view of (\ref{cez}), yields (ii)). To see this, suppose that, on the 
contrary, $\,\xi(u)=0$. Thus, 
$\,(0,u,0)\in\varLambda'\nh
\subseteq\varOmega_0^{\phantom i}\cap\varPhi^{-1}(0)\,$ (see ($*$)) and, by 
(\ref{fsu}), 
$\,(s_j,u_j,x_j)\in\varOmega_0^{\phantom i}\smallsetminus\varPhi^{-1}(0)$. 
Also, $\,(s_j,u_j,x_j)\to(0,u,0)$, so that, if $\,j\,$ 
is sufficiently large, $\,(s_j,u_j,x_j)\,$ must lie in 
$\,U\smallsetminus\varPhi^{-1}(0)$. Hence, according to the last paragraph, 
$\,F(s_j,u_j,x_j)\ne0\,$ for large $\,j$, which contradicts (\ref{fsu}), 
completing the proof.

\section{Con\-for\-mal vector fields}\label{cv}
The symbol $\,\nabla\hs$ always stands both for the Le\-vi-Ci\-vi\-ta 
connection and the gradient operator of a given 
pseu\-\hbox{do\hs-}\hskip0ptRiem\-ann\-i\-an manifold $\,(M,g)$. If $\,v\,$ 
is a vector field on $\,M\nh$, denoting by 
$\,A=\nabla\nh v-[\hh\nabla\nh v]^*$ twice the skew-ad\-joint part of 
$\,\nabla\nh v$, with $\,\nabla\nh v:\tm\to\tm\,$ as in the lines 
following (\ref{cnd}), we can rewrite condition (\ref{lvg}) as
\begin{equation}\label{tnv}
2\hh\nabla\nh v\,\,=\,\,A\,\,+\,\,\phi\hskip2pt\mathrm{Id}\hs.
\end{equation}
We then obviously have $\,e^\tau\nnh\pounds_v(e^{-\tau}\hskip-2ptg)=\pounds_vg
-[(d\hh\tau)(v)]\hs g\,$ for any function $\,\tau:M\to\bbR$, and so, under 
the assumption (\ref{lvg}), the condition $\,(d\hh\tau)(v)=\phi\,$ is 
necessary and sufficient in order that $\,v\,$ be a Kil\-ling field for the 
metric $\,e^{-\tau}\nnh g\,$ con\-for\-mal to $\,g$. At points where 
$\,v\,$ is nonzero, $\,\tau\,$ with $\,(d\hh\tau)(v)=\phi\,$ always exists 
locally, due to solvability of ordinary differential equations. Thus, such 
points are never essential (cf.\ the Introduction).
\begin{lemma}\label{krnll}Let\/ $\,z\in M\,$ be a zero of a con\-for\-mal 
vector field\/ $\,v\,$ on a pseu\-\hbox{do\hs-}\hskip0ptRiem\-ann\-i\-an 
manifold\/ $\,(M,g)$, and let\/ $\,\phi\,$ be the function in\/ 
{\rm(\ref{lvg})}.
\begin{enumerate}
  \def\theenumi{{\rm\alph{enumi}}}
\item[{\rm(a)}] If\/ $\,\phi(z)\ne0$, then\/ 
$\,\mathrm{Ker}\hskip1.7pt\nabla\nh v_z$ is a null subspace of\/ $\,\tzm\nh$.
\item[{\rm(b)}] If\/ $\,\phi(z)=0$, then\/ 
$\,\mathrm{Ker}\hskip1.7pt\nabla\nh v_z$ has even co\-di\-men\-sion in\/ 
$\,\tzm\nh$, and its orthogonal complement is the image\/ 
$\,\nabla\nh v_z(\tzm)$.
\end{enumerate}
\end{lemma}
\begin{proof}This is immediate since, in view of (\ref{tnv}), 
$\,\mathrm{Ker}\hskip1.7pt\nabla\nh v_z\nh\smallsetminus\{0\}\,$ consists 
of eigen\-vectors of the skew-ad\-joint operator $\,A_z:\tzm\to\tzm\,$ for 
the eigen\-value $\,-\hh\phi(z)$.
\end{proof}
It is well-known \cite{cartan-cr,cartan-as,ogiue,capocci} that (\ref{lvg}) 
implies further differential equations. In dimensions $\,n\ge3$, this allows 
us to identify con\-for\-mal vector fields on $\,(M,g)\,$ with parallel 
sections of a certain vector bundle over $\,M\nh$, carrying a natural 
connection; consequently, the dimension of the space of con\-for\-mal fields 
on $\,(M,g)\,$ cannot exceed $\,(n+1)(n+2)/2$. Specifically, if $\,v\,$ 
satisfies (\ref{lvg}) on $\,(M,g)$, and $\,\dim M=n\ge2$, then
\begin{equation}\label{nwn}
\begin{array}{l}
\mathrm{a)}\hskip6pt2\hh\nabla_{\!u}\nabla\nh v
=2\hs R(v\wedge u)+d\hh\phi\nh\otimes\nh u
-g(u,\,\cdot\,)\nnh\otimes\nnh\nabla\nh\phi
+[(d\hh\phi)(u)]\hs\mathrm{Id}\hs,\\
\mathrm{b)}\hskip6pt(1-n/2)[\nabla d\hh\phi](u,u)
=S(u,\nabla_{\!u}v)+S(u,\nabla_{\!u}v)+[\hh\nabla_{\!v}S\hh](u,u)
\end{array}
\end{equation}
for all vector fields $\,u$, where both sides in (\ref{nwn}.a) are bundle 
morphisms $\,\tm\to\tm\nh$, the symbol $\,R\,$ stands for the curvature 
tensor, with the sign convention 
$\,R(w\wedge u)\hh u\hh'\nh=\nabla_{\!u}\nabla_{\!w}u\hh'\nh
-\nabla_{\!w}\nabla_{\!u}u\hh'\nh+\nabla_{\![w,u]}u\hh'$ for vector fields 
$\,w,u,u\hh'\nnh$, and 
$\,S=\mathrm{Ric}-\hs(2n-2)^{-1}\,\sigma g\,$ is the {\it Schouten tensor}, 
with $\,\sigma\,$ denoting the scalar curvature. In coordinates, 
$\,2\hh v^{\hs l}{}_{,\hh kj}\nh=2\hh R_{\hh pjk}{}^{\hh l}v^{\hh p}\nh
+\phi_{\nh,\hh k}^{\phantom i}\delta_{\nh j}^{\hs l}\nh
-\phi^{\hs,\hh l}g_{jk}\nh+\phi_{\nh,\hs j}^{\phantom i}\delta_k^{\hs l}$ and 
$\,(1-n/2)\phi_{\nh,\hs jk}\nh=S_{jp}v^{\hs p}{}_{,\hh k}\nh
+S_{kp}v^{\hs p}{}_{,\hs j}\nh+S_{jk,\hh p}v^{\hh p}\nnh$. In fact, the 
coordinate version of (\ref{nwn}.a) follows from the more general fact that, 
given a $\,1$-form $\,\xi\,$ on a manifold with a tor\-sion\-free connection, 
setting $\,a_{jk}^{\phantom i}\nh=\xi_{\hh k,\hs j}^{\phantom i}\nh
+\xi_{\hh j,\hs k}^{\phantom i}$, one obtains 
$\,\xi_{\hh l,\hs kj}^{\phantom i}=R_{lkj}{}^p\hs\xi_p^{\phantom i}
+(a_{lj,\hs k}^{\phantom i}\nh+a_{lk,\hs j}^{\phantom i}\nh
-a_{kj,\hs l}^{\phantom i})/2$, in view of the Ric\-ci and Bian\-chi 
identities, cf.\ \cite[the bottom of p. 572]{derdzinski-roter}.

Equality (\ref{nwn}.b) can be justified as follows (with (\ref{lvg}) and 
(\ref{nwn}) always meaning the coordinate versions). First, due to the second 
Bian\-chi identity, 
$\,R_{\hh pjk}{}^{\hh l}{}_{,\hs l}=R_{kp,\hh j}\nh-R_{kj,\hh p}$, while the 
Boch\-ner formula (contracted Ric\-ci identity) gives 
$\,v^{\hs l}{}_{,\hh kl}\nh=R_{kp}v^{\hh p}\nh
+n\hs\phi_{,\hh k}/2\,$ (as $\,v^{\hh p}{}_{,\hh p}\nh
=n\hs\phi/2$), and hence $\,v^{\hs l}{}_{,\hh klj}
=R_{kp,\hh j}v^{\hh p}\nh+R_{kp}v^{\hh p}{}_{,\hs j}\nh
+n\hs\phi_{,\hh jk}/2$. Subtracting the Ric\-ci 
identity $\,v^{\hs l}{}_{,\hh kjl}\nh-v^{\hs l}{}_{,\hh klj}\nh
=R_{\hh pjk}{}^{\hh l}v^{\hh p}{}_{,\hh l}\nh+R_{jp}v^{\hh p}{}_{,\hh k}$ from 
$\,1/2\,$ times the formula obtained by applying $\,\nabla_{\!l}^{\phantom i}$ 
to (\ref{nwn}.a), and using the above expressions for 
$\,R_{\hh pjk}{}^{\hh l}{}_{,\hs l}$ and $\,v^{\hs l}{}_{,\hh klj}$, we see 
that 
$\,(1-n/2)\phi_{\nh,\hs jk}\nh=R_{jp}v^{\hs p}{}_{,\hh k}\nh
+R_{kp}v^{\hs p}{}_{,\hs j}\nh+R_{jk,\hh p}v^{\hh p}\nh
+\phi_{\nh,\hh l}{}^lg_{jk}/2$. Now 
(\ref{nwn}.b) easily follows from (\ref{lvg}) since 
$\,R_{jk}\nh=S_{jk}\nh+\hs(2n-2)^{-1}\,\sigma g_{jk}$ and 
$\,(1-n)\hh\phi_{\nh,\hh k}{}^k\nh=\sigma\phi+(d\hh\sigma)(v)$. The last 
relation is another general consequence of (\ref{lvg}): 
$\,\phi_{\nh,\hh k}{}^k\nh=(\phi\hs g_{jk})^{,\hh jk}\nh
=(v_{j,\hh k}\nh+v_{\hh k,\hs j})^{,\hh jk}\nh
=2\hh v^{\hs l}{}_{,\hh kl}{}^k\nh=2(R_{kp}v^{\hh p})^{,\hh k}\nh
+n\hs\phi_{\nh,\hh k}{}^k$ in view of the equality 
$\,v^{\hh j,\hh k}{}_{jk}\nh=v^{\hh j,\hh k}{}_{kj}$ (immediate from the 
Ric\-ci identity) and the Boch\-ner formula just mentioned; on the other hand, 
$\,2(R_{kp}v^{\hh p})^{,\hh k}\nh=
R^{jk}(v_{j,\hh k}\nh+v_{\hh k,\hs j})+2R_{kp}{}^{,\hh k}v^{\hh p}\nh=
\sigma\phi+(d\hh\sigma)(v)$, as the Bian\-chi identity for the Ric\-ci tensor 
gives $\,2R_{kp}{}^{,\hh k}\nh=\sigma_{,\hh p}$.

\section{The case of pseu\-\hbox{do\hs-}\hskip0ptEuclid\-e\-an 
spaces}\label{pe}
Let $\,V\hs$ be an \hbox{$\,n\hh$-}\hskip0ptdi\-men\-sion\-al 
pseu\-\hbox{do\hs-}\hskip0ptEuclid\-e\-an space with the inner product 
$\,\langle\,,\rangle$. For any $\,w,u\in V\nnh$, any skew-ad\-joint 
en\-do\-mor\-phism $\,B:V\to V\nnh$, and $\,c\in \bbR$, the formula
\begin{equation}\label{vxe}
v_x\,=\,\hs w\,+\,Bx\,+\,cx\,+\,2\langle u,x\rangle\hh x\,
-\,\langle x,x\rangle\hh u
\end{equation}
is easily seen to define a con\-for\-mal vector field $\,v\,$ on 
$\,(V\nnh,g)$, where $\,g\,$ is the constant flat metric correspoding to 
$\,\langle\,,\rangle$. If $\,n\ge3$, the resulting vector space of 
con\-for\-mal fields has the maximum possible dimension $\,(n+1)(n+2)/2\,$ 
(See the lines preceding (\ref{nwn}).) Thus, (\ref{vxe}) describes {\it all\/} 
con\-for\-mal fields on any open sub\-man\-i\-fold of $\,(V\nnh,g)$.

Defining $\,v\,$ by (\ref{vxe}) with $\,w=0\,$ and $\,c=0$, we see that 
$\,v=0\,$ everywhere in the set $\,\varPi
=\{x\in\mathrm{Ker}\hskip1.7ptB:\langle u,x\rangle=\langle x,x\rangle=0\}$. 
If, in addition, $\,u\,$ does not lie in the image $\,B(V)$, then all zeros 
$\,x\,$ of $\,v\,$ sufficiently close to $\,0\,$ lie in $\,\varPi$. In fact, 
as $\,0=\langle v_x,x\rangle=\langle u,x\rangle\langle x,x\rangle$, it follows 
that $\,\langle x,x\rangle=0$, or else the equality 
$\,0=v_x\nh=Bx+2\langle u,x\rangle\hh x-\langle x,x\rangle\hh u\,$ with 
$\,\langle u,x\rangle=0\,$ would give $\,u\in B(V)$. Thus, 
$\,0=v_x\nh=Bx+2\langle u,x\rangle\hh x$. Such $\,x\,$ which also have the 
property that $\,\langle u,x\rangle\ne0\,$ cannot be arbitrarily close to 
$\,0$, since they all lie in hyperplanes given by 
$\,2\langle u,x\rangle=-\hs b$, where $\,b\,$ ranges over nonzero eigenvalues 
of $\,B$. Consequently, $\,\langle u,x\rangle=\langle x,x\rangle=0\,$ for all 
zeros of $\,v\,$ near $\,0$, and then also $\,0=v_x\nh=Bx$.
\begin{example}\label{gensg}
For $\,v\,$ as in the last paragraph, let $\,Z'$ be the connected 
component of the zero set of $\,v\,$ containing $\,\varPi\nnh$. If $\,B\,$ and 
$\,u\,$ are chosen so that the restriction of $\,\langle\,,\rangle\,$ to 
$\,u^\perp\nh\cap\mathrm{Ker}\hskip1.7ptB\,$ is not sem\-i\-def\-i\-nite, 
$\,Z'$ will have a singularity at $\,0\,$ (Remark~\ref{snglr}(a)). On the 
other hand, sem\-i\-def\-i\-nite\-ness of $\,\langle\,,\rangle\,$ on 
$\,u^\perp\nh\cap\mathrm{Ker}\hskip1.7ptB\,$ implies that $\,\varPi\,$ is a 
sub\-man\-i\-fold of $\,V\hs$ (see Remark~\ref{semdf}), while, if 
$\,\langle\,,\rangle\,$ is indefinite, one can also choose such $\,B\,$ and 
$\,u\,$ for which, in addition, $\,\dim\varPi\ge1\,$ and 
$\,\dim V\nnh-\dim\varPi\hs$ is odd.
\end{example}
\begin{example}\label{srfcs}In $\,\rto$ with the Cartesian coordinates 
$\,x^{\hs j}\nnh$, let $\,g\,$ be the flat metric given by 
$\,g_{12}^{\phantom i}\nh=g_{21}^{\phantom i}\nh=1\,$ and 
$\,g_{11}^{\phantom i}\nh=g_{22}^{\phantom i}\nh=0$. The con\-for\-mal vector 
fields $\,v\,$ for $\,g\,$ are obviously characterized by the 
par\-tial de\-riv\-a\-tive conditions 
$\,\partial_1^{\phantom i}v_1^{\phantom i}\nh
=\partial_2^{\phantom i}v_2^{\phantom i}\nh=0$, that is, 
$\,\partial_1^{\phantom i}v^2\nnh=\partial_2^{\phantom i}v^1\nnh=0$. Hence 
$\,v^1$ may be any function of $\,x^1$ and $\,v^2$ any function of 
$\,x^2\nnh$. Thus, the zero set of $\,v\,$ can have the form 
$\,\varXi\times\varXi'\nnh$, with any closed sets 
$\,\varXi,\varXi'\subseteq\bbR$. 
\end{example}

\section{Intermediate sub\-man\-i\-folds}\label{is}
The proof of Theorem~\ref{nbzrs} under the assumption (\ref{cnd}.b), given in 
Section~\ref{cb}, uses a sub\-man\-i\-fold $\,N\hs$ containing all zeros of a 
given con\-for\-mal vector field that lie near a prescribed zero $\,z$, and 
having the tangent space $\,\mathrm{Ker}\hskip1.7pt\nabla\nh v_z$ at $\,z$. 
According to Example~\ref{zrset}, such $\,N\hs$ always exists. For easy 
reference, this fact and some properties of $\,N\hs$ are gathered in the 
following lemma. Radial limit directions of a set are defined at the end of 
Section~\ref{cl}.
\begin{lemma}\label{intsb}Given a con\-for\-mal vector field\/ $\,v\,$ on a 
pseu\-\hbox{do\hs-}\hskip0ptRiem\-ann\-i\-an manifold\/ $\,(M,g)\,$ and a 
zero\/ $\,z\in M\hs$ of $\,v$, there exists a sub\-man\-i\-fold\/ $\,N$ of\/ 
$\,M\hs$ such that
\begin{enumerate}
  \def\theenumi{{\rm\roman{enumi}}}
\item[{\rm(i)}] all zeros of\/ $\,v\,$ sufficiently close to\/ $\,z\,$ lie 
in\/ $\,N\nnh$,
\item[{\rm(ii)}] $\txn\nh=\hh\mathrm{Ker}\hskip1.7pt\nabla\nh v_x$ at every\/ 
$\,x\in N$ at which\/ $\,v_x\nh=0\,$ and 
$\,\mathrm{rank}\hskip2.9pt\nabla\nh v_x\nh
=\mathrm{rank}\hskip2.9pt\nabla\nh v_z$,
\item[{\rm(iii)}] all radial limit directions of the zero set of\/ $\,v\,$ 
at\/ $\,z\,$ lie in\/ $\,\tzn\nh=\hh\mathrm{Ker}\hskip1.7pt\nabla\nh v_z$,
\item[{\rm(iv)}] if\/ $\,\mathrm{rank}\hskip2.9pt\nabla\nh v_x\nh
=\mathrm{rank}\hskip2.9pt\nabla\nh v_z$ at a point\/ $\,x\in N$ at which\/ 
$\,v_x\nh=0$, and\/ $\,w\,$ is a vector field on $\,M\hs$ with\/ 
$\,w_x\nh\in\hh\mathrm{Ker}\hskip1.7pt\nabla\nh v_x$, while\/ $\,\phi(x)=0$, 
for\/ $\,\phi\,$ in\/ {\rm(\ref{lvg})}, then the function\/ 
$\,Q=2\hh g(w,v)\,$ restricted to\/ $\,N\hs$ has a critical point at $\,x$, 
and its Hess\-i\-an\/ $\,\partial\hs d\hh Q_x$ equals the right-hand side in 
Theorem\/~{\rm\ref{tgdir}(d)} with\/ $\,\langle\,,\rangle=g_x$ and\/ $\,w=w_x$.
\end{enumerate}
\end{lemma}
\begin{proof}For $\,N\hs$ constructed in Example~\ref{zrset}, with 
$\,\mathcal{E}=\tm\,$ and $\,\psi=v$, one clearly has (i) and (ii), while 
(iii) follows from Remark~\ref{submf}(i). Finally, in (iv), the condition 
$\,w_x\nh\in\mathrm{Ker}\hskip1.7pt\nabla\nh v_x$ implies, by 
Lemma~\ref{krnll}(b), that the value $\,\xi_x$ of the $\,1$-form 
$\,\xi=2\hh g(w,\,\cdot\,)\,$ vanishes on the image 
$\,\nabla\nh v_x(\txm)$. Now (iv) is obvious from the Hess\-i\-an formula in 
Remark~\ref{hessn}(ii), combined with the expression for the second covariant 
derivative of $\,v\,$ at $\,x\,$ provided by (\ref{nwn}.a), in which the 
curvature term vanishes since $\,v_x\nh=0$.
\end{proof}

\section{Con\-for\-mal fields along geodesics}\label{cg}
Given a con\-for\-mal vector field $\,v\,$ on a 
pseu\-\hbox{do\hs-}\hskip0ptRiem\-ann\-i\-an manifold $\,(M,g)\,$ of dimension 
$\,n\ge3$, with (\ref{lvg}), let us consider a parallel vector field 
$\,t\mapsto u(t)\in T_{x(t)}M\,$ along a geodesic $\,t\mapsto x(t)\,$ of 
$\,(M,g)$. Transvecting the coordinate versions of (\ref{lvg}) and (\ref{nwn}) 
with $\,\dot x^{\hh j}\dot x^{\hs k}$ or, respectively, 
$\,\dot x^{\hh j}u^{\hs k}\nnh$, we obtain
\begin{equation}\label{vdx}
\begin{array}{rl}
\mathrm{i)}&\hskip3pt2\hh\langle v,\dot x\rangle\dot{\,}\hs
=\,\phi\langle\dot x,\dot x\rangle\hs,\\
\mathrm{ii)}&\hskip3pt2\hh\nabla_{\!\dot x}\nabla_{\!u}v
=2\hs R(v\wedge \dot x)\hh u+[(d\hh\phi)(u)]\hs\dot x
+\hskip3.3pt\dot{\hskip-3.3pt\phi}\hh u
-\langle\dot x,u\rangle\nabla\nh\phi\hs,\\
\mathrm{iii)}&\hskip3pt(1-n/2)[(d\hh\phi)(u)]\hs\dot{\,}
=S(u,\nabla_{\!\dot x}v)+S(\dot x,\nabla_{\!u}v)
+[\hh\nabla_{\!v}S\hh](u,\dot x)\hs,
\end{array}
\end{equation}
where $\,(\hskip1.7pt,\hskip1pt)\dot{\,}\nh=\hh d/dt$, 
cf.\ (\ref{dfe}), the symbol $\,\langle\,,\rangle\,$ stands for $\,g$, and the 
dependence of both sides on $\,t\,$ suppressed in the notation: 
$\,v=v_{x(t)},\hs\phi=\phi(x(t))$. With $\,u=\dot x$, (\ref{vdx}) gives
\begin{equation}\label{ndx}
\begin{array}{rl}
\mathrm{i)}&\hskip3pt\nabla_{\!\dot x}\nabla_{\!\dot x}v
=R(v\wedge \dot x)\hh\dot x+\hskip2.8pt\dot{\hskip-3.3pt\phi}\hs\dot x
-\langle\dot x,\dot x\rangle\nabla\nh\phi/2\hs,\\
\mathrm{ii)}&\hskip3pt(1-n/2)\hh\hskip3.3pt\ddot{\hskip-3.3pt\phi}
=2S(\dot x,\nabla_{\!\dot x}v)+[\hh\nabla_{\!v}S\hh](\dot x,\dot x)\hs.
\end{array}
\end{equation}
As a consequence of (\ref{ndx}.i), condition (\ref{lvg}) implies that
\begin{equation}\label{nwd}
\nabla_{\!\dot x}\nabla_{\!\dot x}(v\wedge\dot x)
=[R(v\wedge\dot x)\hh\dot x]\wedge\dot x\hskip12pt\mathrm{\ if\ }
\,t\mapsto x(t)\,\mathrm{\ is\ a\ null\ geodesic.}
\end{equation}
If, in addition, $\,v\,$ is tangent to a null geodesic $\,t\mapsto x(t)$, that 
is, $\,v_{x(t)}$ is a multiple of $\,\dot x(t)\,$ for every $\,t$, then, by 
(\ref{ndx}.i) with $\,v\wedge\dot x=0\,$ and $\,\langle\dot x,\dot x\rangle=0$,
\begin{equation}\label{nvf}
\nabla_{\!\dot x}\nabla_{\!\dot x}v\,\,
\,=\,\hskip2.8pt\dot{\hskip-3.3pt\phi}\hs\dot x\hs.
\end{equation}
\begin{remark}\label{phiez}If two distinct zeros\/ $\,z,x\,$ of a 
con\-for\-mal vector field\/ $\,v\,$ on a 
pseu\-\hbox{do\hs-}\hskip0ptRiem\-ann\-i\-an manifold\/ $\,(M,g)\,$ are 
joined by a non-null geodesic segment $\,\varGamma$ and\/ $\,\phi\,$ is the 
function in\/ {\rm(\ref{lvg})}, then\/ $\,\phi=0\,$ somewhere in\/ 
$\,\varGamma\smallsetminus\{z,x\}$. 

This is clear since, in (\ref{vdx}.i), $\,\langle v,\dot x\rangle=0\,$ at both 
ends of the parameter interval.
\end{remark}
In the next lemma, $\,T\hskip-3.1pt_\varGamma^{\phantom i}\nh M\,$ denotes 
the restriction of $\,\tm\,$ to the \hbox{one\hh-}\hskip0ptdi\-men\-sion\-al 
null immersed sub\-man\-i\-fold $\,\varGamma\nnh$. Thus, $\,T\nh\varGamma$ 
and $\,(T\nh\varGamma)^\perp$ are sub\-bun\-dles of 
$\,T\hskip-3.1pt_\varGamma^{\phantom i}\nh M\nh$, while $\,g\,$ obviously 
induces a (nondegenerate) fibre metric in the quotient bundle 
$\,(T\nh\varGamma)^\perp\nnh/\hh(T\nh\varGamma)$. By 
$\,\mathfrak{conf}\,[(T\nh\varGamma)^\perp\nnh/\hh(T\nh\varGamma)]\,$ we 
denote the vector bundle over $\,\varGamma$ whose sections are 
infinitesimal con\-for\-mal en\-do\-mor\-phism of 
$\,(T\nh\varGamma)^\perp\nnh/\hh(T\nh\varGamma)$. (An {\it infinitesimal 
con\-for\-mal en\-do\-mor\-phism\/} is one with the self-ad\-joint part equal 
to a function times $\,\mathrm{Id}$.)
\begin{lemma}\label{nvprl}Let a con\-for\-mal vector field\/ $\,v\hs$ on a 
pseu\-\hbox{do\hs-}\hskip0ptRiem\-ann\-i\-an manifold\/ $\,(M,g)\,$ be 
tangent to a nontrivial null geodesic\/ $\,\varGamma$ with a parametrization 
$\,[\hs0,1\hh]\ni t\mapsto x(t)$, so that\/ $\,x(0)=y\,$ and\/ 
$\,\nabla_{\!\dot x}v=\lambda\dot x\,$ at\/ $\,t=0$, for some\/ $\,y\in M\hs$ 
and\/ $\,\lambda\in\bbR$.
\begin{enumerate}
  \def\theenumi{{\rm\alph{enumi}}}
\item[{\rm(a)}] Along\/ $\,\varGamma\nnh$, we have\/ 
$\,\nabla_{\!\dot x}v=[\lambda+\phi-\phi(y)]\hh\dot x$, with\/ $\,\phi\,$ as 
in\/ {\rm(\ref{lvg})}.
\item[{\rm(b)}] If\/ $\,\phi\,$ is constant along\/ $\,\varGamma$ and\/ 
$\,n=\dim M\ge2$, then\/ $\,\nabla\nh v$, restricted to\/ $\,\varGamma\nnh$,
\begin{enumerate}
  \def\theenumi{{\rm\roman{enumi}}}
\item[{\rm(i)}] acts on\/ $\,T\nh\varGamma$ and\/ 
$\,(T\hskip-3.1pt_\varGamma^{\phantom i}\nh M)/(T\nh\varGamma)^\perp$ as the 
multiplications by\/ $\,\lambda\,$ and\/ $\,\phi\hs-\nh\lambda$,
\item[{\rm(ii)}] descends to a parallel section of\/ 
$\,\mathfrak{conf}\,[(T\nh\varGamma)^\perp\nnh/\hh(T\nh\varGamma)]\,$ with 
the trace\/ $\,(n-2)\hh\phi/2$,
\item[{\rm(iii)}] has the same characteristic polynomial at all points of\/ 
$\,\varGamma\nnh$.
\end{enumerate}
\end{enumerate}
\end{lemma}
\begin{proof}Integrating (\ref{nvf}), we obtain (a), while (b\hs-i) for 
$\,T\nh\varGamma$ is obvious from (a) with $\,\phi=\phi(y)$. Next, in view of 
(\ref{tnv}) with $\,A\nnh^*\nh=-\hs A$, the sub\-bun\-dles $\,T\nh\varGamma$ and 
$\,(T\nh\varGamma)^\perp$ of $\,T\hskip-3.1pt_\varGamma^{\phantom i}\nh M\,$ 
are $\,\nabla\nh v$-in\-var\-i\-ant; $\,\nabla\nh v$-in\-var\-i\-ance of 
the latter follows from that of the former, since it is the same as 
$\,A\nh$-in\-var\-i\-ance. Thus, $\,\nabla\nh v\,$ descends to an 
en\-do\-mor\-phism of $\,(T\nh\varGamma)^\perp\nnh/\hh(T\nh\varGamma)$, which 
is obviously con\-for\-mal and has the trace claimed in (b\hs-ii), since, by 
(\ref{lvg}), $\,\nabla\nh v+[\hh\nabla\nh v]^*\nh=\phi\hskip2pt\mathrm{Id}$. 

Let $\,t\mapsto u(t)\,$ be any vector field along $\,\varGamma\nnh$. Using 
(\ref{tnv}) with $\,A\nnh^*\nh=-\hs A\,$ and (a) with $\,\phi=\phi(y)$, we 
see that 
$\,2\langle\dot x,\nabla_{\!u}v\rangle=\langle\dot x,Au+\phi\hh u\rangle
=\langle\phi\hh\dot x-A\dot x,u\rangle
=2\langle\phi\hh\dot x-\nabla_{\!\dot x}v,u\rangle
=2\langle\dot x,(\phi\hs-\nh\lambda)\hh u\rangle$, which proves (b\hs-i) for 
$\,(T\hskip-3.1pt_\varGamma^{\phantom i}\nh M)/(T\nh\varGamma)^\perp\nnh$. 
Now, if $\,\nabla_{\!\dot x}u=0\,$ and $\,\langle\dot x,u\rangle=0$, 
(\ref{vdx}.ii) implies that 
$\,2\hh\nabla_{\!\dot x}\nabla_{\!u}v=[(d\hh\phi)(u)]\hs\dot x$, as our 
assumptions give $\,v\wedge \dot x=0\,$ and 
$\,\hskip3.3pt\dot{\hskip-3.3pt\phi}=0$. Thus, $\,\nabla_{\!u}v\,$ projects 
onto a parallel section of $\,(T\nh\varGamma)^\perp\nnh/\hh(T\nh\varGamma)$, 
which proves (b\hs-ii). 

Due to $\,\nabla\nh v$-in\-var\-i\-ance of the sub\-bun\-dles 
$\,T\nh\varGamma\nnh,\,(T\nh\varGamma)^\perp$ of 
$\,T\hskip-3.1pt_\varGamma^{\phantom i}\nh M\,$ and the inclusion 
$\,T\nh\varGamma\subseteq(T\nh\varGamma)^\perp\nnh$, the characteristic 
polynomial of $\,\nabla\nh v\,$ in 
$\,T\hskip-3.1pt_\varGamma^{\phantom i}\nh M\,$ equals the product of the 
characteristic polynomials of the en\-do\-mor\-phisms induced by 
$\,\nabla\nh v\,$ in the three bundles 
$\,T\nh\varGamma\nnh$, $\,(T\nh\varGamma)^\perp\nnh/\hh(T\nh\varGamma)$ and 
$\,(T\hskip-3.1pt_\varGamma^{\phantom i}\nh M)/(T\nh\varGamma)^\perp\nnh$. 
By (b\hs-i), the first and last of these are polynomials of degree one with 
the roots $\,\lambda\,$ and $\,\phi\hs-\nh\lambda$, constant along 
$\,\varGamma\nnh$, while the second polynomial is constant along 
$\,\varGamma$ as a consequence of (b\hs-ii), which completes the proof.
\end{proof}
\begin{lemma}\label{zeros}Let there be given a con\-for\-mal vector field\/ 
$\,v\,$ on a pseu\-\hbox{do\hs-}\hskip0ptRiem\-ann\-i\-an manifold\/ 
$\,(M,g)\,$ of dimension\/ $\,n\ge3$, a point\/ $\,z\in M\nh$, a nonzero 
null vector\/ $\,w\in\tzm\nh$, and a nontrivial geodesic segment\/ 
$\,\varGamma\,$ in\/ $\,(M,g)\,$ containing\/ $\,z\,$ and\/ tangent to\/ 
$\,w\,$ at\/ $\,z$.
\begin{enumerate}
  \def\theenumi{{\rm\alph{enumi}}}
\item[{\rm(a)}] If\/ $\,v_z\nh=0\,$ and\/ 
$\,w\in\hh\mathrm{Ker}\hskip1.7pt\nabla\nh v_z
\cap\hh\mathrm{Ker}\hskip2.2ptd\hh\phi_z$, for the function\/ $\,\phi\,$ in\/ 
{\rm(\ref{lvg})}, then\/ $\,v=0\,$ and\/ $\,\phi=\phi(z)\,$ at every point 
of\/ $\,\varGamma\nnh$.
\item[{\rm(b)}] If\/ $\,v_z\nh=0\,$ and\/ $\,v_x\nh=0\,$ for some point\/ 
$\,x\in\varGamma\nh\smallsetminus\{z\}\,$ lying in a 
nor\-mal-co\-or\-di\-nate neighborhood\/ $\,\,U'$ of\/ $\,z$, or\/ 
$\,v_z$ and\/ $\,\nabla_{\!w}v\in\tzm\,$ are both tangent to\/ 
$\,\varGamma\,$ at\/ $\,z$, then\/ $\,v\,$ is tangent to\/ $\,\varGamma\,$ at 
every point of\/ $\,\varGamma\nnh$.
\end{enumerate}
\end{lemma}
\begin{proof}Let $\,t\mapsto x(t)\,$ be a geodesic pa\-ram\-e\-tri\-za\-tion 
of $\,\varGamma\,$ with $\,x(0)=z$. As $\,\langle\dot x,\dot x\rangle=0$, 
(\ref{ndx}) constitutes a system of first-or\-der linear homogeneous ordinary 
differential equations with the unknowns $\,v,\nabla_{\!\dot x}v\,$ and 
$\,\hskip2.8pt\dot{\hskip-3.3pt\phi}$, for which the assumption in (a) amount 
to choosing the zero initial conditions at $\,t=0$. The conclusion of (a) is 
now obvious from uniqueness of solutions. Similarly, if $\,v\,$ and 
$\,\nabla_{\!\dot x}v\,$ are both tangent to the geodesic at $\,t=0$, 
(\ref{nwd}) implies, for the same reason, that $\,v\wedge\dot x=0\,$ at every 
$\,t$, proving (b) in this case. Finally, if $\,v_z\nh=0=v_x$ for $\,x\,$ as 
in (b), the local flow of $\,v\,$ sends the portion of $\,\varGamma\,$ joining 
$\,z\,$ to $\,x\,$ into itself (Remark~\ref{nlseg}). Combined with the 
preceding sentence, this shows that $\,v\,$ is tangent to $\,\varGamma\nnh$.
\end{proof}
\begin{lemma}\label{zrnlg}Suppose that\/ $\,[\hs0,1\hh]\ni t\mapsto x(t)\,$ is 
a parametrization of a nontrivial null geodesic segment\/ $\,\varGamma$ with 
the endpoints\/ $\,z=x(0)\,$ and\/ $\,y=x(1)\,$ in a 
pseu\-\hbox{do\hs-}\hskip0ptRiem\-ann\-i\-an manifold\/ $\,(M,g)$, while\/ 
$\,v\,$ is a con\-for\-mal vector field on\/ $\,(M,g)\,$ vanishing at both\/ 
$\,z\,$ and\/ $\,x$. For\/ $\,\phi\,$ and\/ 
$\,\hskip2.8pt\dot{\hskip-3.3pt\phi}\,$ as in\/ {\rm(\ref{lvg})} and\/ 
{\rm(\ref{dfe})}, one has
\begin{enumerate}
  \def\theenumi{{\rm\roman{enumi}}}
\item[{\rm(i)}] if\/ $\,v\,$ is tangent to\/ $\,\varGamma\nnh$, then 
$\,\nabla_{\!\dot x}v=0\,$ somewhere in\/ 
$\,\varGamma\nh\smallsetminus\{z,y\}$,
\hskip3ptand
\item[{\rm(ii)}] if\/ $\,\nabla_{\!\dot x}v=0\,$ at\/ $\,z$, then\/ 
$\,\phi-\phi(z)\,$ and\/ $\,\hskip2.8pt\dot{\hskip-3.3pt\phi}\,$ vanish at 
some points of\/ $\,\varGamma\nh\smallsetminus\{z,y\}$.
\end{enumerate}
\end{lemma}
\begin{proof}By Lemma~\ref{zeros}(b), $\,v\,$ is tangent to $\,\varGamma\hs$ 
in case (ii) as well. Since $\,v_{x(t)}$ is a multiple of $\,\dot x(t)$, it 
may be viewed as a function $\,[\hs0,1\hh]\to\bbR$, equal to $\,0\,$ at the 
endpoints. Its derivative $\,\nabla_{\!\dot x}v\,$ therefore vanishes at some 
$\,t\in(0,1)$. Under the assumption of (ii), $\,\nabla_{\!\dot x}v=0\,$ both 
at $\,0\,$ and somewhere in $\,(0,1)$, so that (\ref{nvf}) (for the second 
derivative $\,\nabla_{\!\dot x}\nabla_{\!\dot x}v$) and the equality 
$\,\nabla_{\!\dot x}v=[\hh\phi-\phi(z)]\hs\dot x$, due to 
Lemma~\ref{nvprl}(a), imply (ii).
\end{proof}

\section{The first of the two inclusions}\label{fi}
Assertion (a) in the following lemma implies a part of the conclusion in 
Theorem~\ref{nbzrs}. See Remark~\ref{incls} for details.
\begin{lemma}\label{cstdm}For a con\-for\-mal vector field\/ $\,v\,$ on a 
pseu\-\hbox{do\hs-}\hskip0ptRiem\-ann\-i\-an manifold\/ 
$\,(M,g)\,$ of dimension\/ $\,n\ge3$, a point\/ $\,z\in M\,$ with\/ 
$\,v_z\nh=0$, and\/ $\,\phi\,$ as in\/ {\rm(\ref{lvg})}, let us set\/ 
$\,V\nnh=H\cap H^\perp$ and $\,E=C\cap H$, where\/ 
$\,H=\hs\mathrm{Ker}\hskip1.7pt\nabla\nh v_z
\cap\hs\mathrm{Ker}\hskip2.2ptd\hh\phi_z$ and\/ $\,C=\{u\in\tzm:g_z(u,u)=0\}\,$ 
denotes the null cone, so that $\,V$ is a null vector subspace of\/ $\,\tzm\,$ 
contained in the subset\/ $\,E$. Whenever\/ $\,\,U\hs$ is a sufficiently small 
star-shap\-ed neighborhood of\/ $\,0\,$ in\/ $\,\tzm\,$ mapped by\/ 
$\,\mathrm{exp}\hh_z$ dif\-feo\-mor\-phi\-cal\-ly onto a neighborhood of\/ 
$\,z\,$ in\/ $\,M\nh$, the image\/ 
$\,K\hs\,=\,\hs\mathrm{exp}\hh_z[\hh V\nnh\cap U\hh]\,$ is a sub\-man\-i\-fold 
of\/ $\,M\nh$, while
\begin{enumerate}
  \def\theenumi{{\rm\alph{enumi}}}
\item[{\rm(a)}] $v=0\,$ and\/ $\,\phi=\phi(z)\,$ everywhere in\/ 
$\,\mathrm{exp}\hh_z[E\cap U\hh]$, and hence everywhere in\/ $\,K$,
\item[{\rm(b)}] for any\/ $\,x\in K$, the parallel transport along the 
geodesic contained in\/ $\,K\,$ which joins\/ $\,z\,$ to\/ $\,x\,$ sends\/ 
$\,H\,$ onto\/ $\,\mathcal{H}_x=\hs\mathrm{Ker}\hskip1.7pt\nabla\nh v_x
\cap\mathrm{Ker}\hskip2.2ptd\hh\phi_x$,
\item[{\rm(c)}] $\dim\hs\mathcal{H}_x$ is constant as a function of\/ 
$\,x\in K$,
\item[{\rm(d)}] if\/ $\,\phi(z)=0$, then\/ 
$\,\mathrm{rank}\hskip2.9pt\nabla\nh v_x$ and\/ 
$\,\dim\hs\mathrm{Ker}\hskip1.7pt\nabla\nh v_x$ are constant as functions of\/ 
$\,x\in K$.
\end{enumerate}
\end{lemma}
\begin{proof}Lemma~\ref{zeros}(a) implies (a). Next, if 
$\,t\mapsto u(t)\in T_{x(t)}M\,$ is a parallel vector field along a geodesic 
$\,t\mapsto x(t)\,$ in $\,K\,$ with $\,x(0)=z\,$ and $\,u(0)\in H$, the 
equality $\,V\nnh=H\cap H^\perp$ gives $\,\langle\dot x,u\rangle=0\,$ at 
$\,t=0\,$ and, consequently, for all $\,t$. 
(As usual, $\,\langle\,,\rangle\,$ stands for $\,g$.) By (a), 
$\,v=\nabla_{\!\dot x}v=0\,$ and $\,\hskip2.8pt\dot{\hskip-3.3pt\phi}=0\,$ 
along the whole geodesic, so that equations (\ref{vdx}.ii) -- (\ref{vdx}.iii) 
now read $\,2\hh\nabla_{\!\dot x}\nabla_{\!u}v
=[(d\hh\phi)(u)]\hs\dot x\,$ and 
$\,(1-n/2)[(d\hh\phi)(u)]\hs\dot{\,}=S(\dot x,\nabla_{\!u}v)$. This is a 
system of first-or\-der linear homogeneous ordinary differential equations 
with the 
unknowns $\,\nabla_{\!u}v\,$ and $\,(d\hh\phi)(u)$, which equal zero at 
$\,t=0\,$ and, therefore, at every $\,t$. The parallel transport along the 
geodesic from $\,z\,$ to $\,x=x(t)\,$ thus sends $\,\mathcal{H}_z$ {\it 
into\/} $\,\mathcal{H}_x$. The word `into' can be replaced with `onto' if 
$\,\,U\hs$ is small enough. Namely, $\,\mathcal{H}_x$ is the kernel of a 
linear operator depending continuously on $\,x\in K$, and so 
$\,\dim\hs\mathcal{H}_x$ is sem\-i\-con\-tin\-u\-ous: 
$\,\dim\hs\mathcal{H}_x\nh\le\dim\hs\mathcal{H}_z$ for $\,x\,$ near $\,z\,$ in 
$\,K$. However, the ``into'' conclusion established above gives 
$\,\dim\hs\mathcal{H}_x\nh\ge\dim\hs\mathcal{H}_z$. Now (b) and (c) follow. 
Finally, $\,p_z\nh-1\le p_x\le p_z$ for 
$\,p_x\nh=\dim\hs\mathrm{Ker}\hskip1.7pt\nabla\nh v_x$ and all $\,x\in K\,$ 
close to $\,z$. In fact, $\,p_x\le p_z$ due, again, to 
sem\-i\-con\-ti\-nu\-i\-ty, and 
$\,p_z\nh-1\le\dim\hs\mathcal{H}_z\nh=\dim\hs\mathcal{H}_x\nh\le p_x$ by (c). 
On the other hand, if $\,\phi=0\,$ at $\,z$, (a) gives $\,\phi=0\,$ on $\,K$. 
Thus, $\,n-p_x$ is even (Lemma~\ref{krnll}(b)), which, combined with the 
inequality $\,p_z\nh-1\le p_x\le p_z$, yields (d).
\end{proof}
\begin{remark}\label{incls}Proving Theorem~\ref{nbzrs} has now been reduced to 
showing that 
\begin{equation}\label{ncl}
Z\cap U'\hs\subseteq\,\hh\mathrm{exp}\hh_z[\hh C\cap H\cap U\hh]\hskip7pt
\mathrm{for\ sufficiently\ small\ }\hs\,U\hs\mathrm{\ and\ }\hs\,U'\nh,
\end{equation}
since the opposite inclusion is provided by Lemma~\ref{cstdm}(a).
\end{remark}

\section{Connecting limits for the zero set}\label{cz}
This section consists of two lemmas needed in the proof of 
Theorem~\ref{nbzrs}. For the definition of 
$\,\bbL_z(Z\smallsetminus\phi^{-1}(0),K)$, see Section~\ref{cl}.
\begin{lemma}\label{tadir}For\/ $\,M,g,v,Z,z,\phi\,$ and\/ $\,H\,$ as in 
Theorem\/~{\rm\ref{nbzrs}}, let\/ $\,V=H\cap H^\perp\nnh$. If\/ 
$\,\phi(z)=0\,$ and\/ $\,\nabla\nh\phi_z\notin\nabla\nh v_z(\tzm)$, while no 
element of\/ $\,\bbL_z(Z\smallsetminus\phi^{-1}(0),K)\,$ is contained in\/ 
$\,(H\smallsetminus V\hh)\cup\{0\}$, then\/ $\,\phi=0\,$ at every zero of\/ 
$\,v\,$ sufficiently close to\/ $\,z$.
\end{lemma}
\begin{proof}We fix a nor\-mal-co\-or\-di\-nate neighborhood $\,\,U'$ of 
$\,z\,$ which is {\it sub\-con\-vex\/} (Section~\ref{nc}), a star-shap\-ed 
neighborhood $\,\,U\,$ of $\,0\,$ in $\,\tzm\,$ mapped by 
$\,\mathrm{exp}\hh_z$ dif\-feo\-mor\-phi\-cal\-ly onto $\,\,U'\nnh$, and a 
Riemannian metric $\,h\,$ on $\,\,U'\nnh$. By Lemma~\ref{cstdm}(a), 
$\,K=\hh\mathrm{exp}\hh_z[\hh V\nnh\cap U\hh]\,$ is a sub\-man\-i\-fold of 
$\,\,U'$ contained in both $\,\phi^{-1}(0)\,$ and the zero set $\,Z\,$ of 
$\,v$. In view of the inverse mapping theorem, applied to the 
$\,g$-ex\-po\-nen\-tial mapping of the $\,h$-nor\-mal bundle of $\,K\,$ in 
$\,\,U'\nnh$, by making $\,\,U'$ smaller we can also ensure that every 
$\,y\in U'\hh\smallsetminus K\,$ is joined to some point $\,p_y\in K\,$ by a 
nontrivial $\,g$-ge\-o\-des\-ic segment $\,\varGamma\hskip-3pt_y$ contained in 
a nor\-mal-co\-or\-di\-nate neighborhood of $\,p_y$ and $\,h$-nor\-mal to 
$\,K\,$ at $\,p_y$.

Furthermore, $\,\phi\ne0\,$ everywhere in 
$\,\varGamma\hskip-3pt_y\smallsetminus\{y,p_y\}$, for all 
$\,y\in(Z\cap U'\hh)\smallsetminus\phi^{-1}(0)$, as long as $\,\,U'$ is 
sufficiently small. Namely, if this were not the case, there would exist a 
sequence of points $\,y\in(Z\cap U'\hh)\smallsetminus\phi^{-1}(0)\,$ 
converging to $\,z\,$ with $\,\phi=0\,$ at some interior point of each 
$\,\varGamma\hskip-3pt_y$. As $\,\phi=0\,$ at the endpoint $\,p_y$ due to the 
inclusion $\,K\subseteq\phi^{-1}(0)$, the 
tangent direction of each $\,\varGamma\hskip-3pt_y$ at some other interior 
point would thus be contained in $\,\mathrm{Ker}\hskip2.2ptd\hh\phi$. In 
view of Lemma~\ref{sbcnv} and Remark~\ref{submf}(ii), a subsequence of such a 
sequence of interior tangent directions would converge to an element $\,L\,$ 
of $\,\bbL_z(Z\smallsetminus\phi^{-1}(0),K)\,$ contained in both 
$\,\mathrm{Ker}\hskip1.7pt\nabla\nh v_z$ and 
$\,\mathrm{Ker}\hskip2.2ptd\hh\phi_z$, so that 
$\,L\subseteq H=\hs\mathrm{Ker}\hskip1.7pt\nabla\nh v_z
\cap\hs\mathrm{Ker}\hskip2.2ptd\hh\phi_z$. By 
Lemma~\ref{sbcnv}, the tangent direction of $\,\varGamma\hskip-3pt_y$ at 
the endpoint $\,p_y$ would tend to $\,L\,$ as well, so that $\,L\,$ would be 
$\,h$-nor\-mal to $\,K\,$ at $\,z$, and hence not contained in $\,V\nnh=\tzk$, 
contrary to our assumption.

We now show that 
$\,(Z\cap U'\hh)\smallsetminus\phi^{-1}(0)=$\hskip3.5pt{\rm\O} \hskip2.6ptfor 
small enough $\,\,U'\nnh$, as required. In fact, otherwise we might fix a 
sequence of distinct points $\,y\in(Z\cap U'\hh)\smallsetminus\phi^{-1}(0)\,$ 
converging to $\,z$. In view of the last paragraph and Remark~\ref{phiez}, 
each of the geodesic segments $\,\varGamma\hskip-3pt_y$ is null. Since the 
null geodesic segment $\,\varGamma\hskip-3pt_y$ lies in a 
nor\-mal-co\-or\-di\-nate neighborhood of $\,p_y$, while $\,v=0\,$ at both 
$\,p_y$ and $\,y$, Lemma~\ref{zeros}(b) implies that $\,v$, and hence 
$\,\nabla_{\!\dot x}v$, is tangent to $\,\varGamma\hskip-3pt_y$. 
In terms of a geodesic pa\-ram\-e\-tri\-za\-tion 
$\,[\hs0,1\hh]\ni t\mapsto x(t)\,$ of $\,\varGamma\hskip-3pt_y$ such that 
$\,x(0)=y$, we thus have, by Lemma~\ref{nvprl}(a), 
$\,\nabla_{\!\dot x}v=[\lambda+\phi-\phi(y)]\hh\dot x\,$ along 
$\,\varGamma\hskip-3pt_y$, with some $\,\lambda\in\bbR$. As 
$\,p_y\in K\subseteq\phi^{-1}(0)$, we have $\,\phi(x(1))=\phi(p_y)=0$, and so 
$\,\dot x(1)\,$ is an eigen\-vector of $\,\nabla\nh v_{x(1)}$ for the 
eigen\-value $\,\lambda_y\nh=\lambda-\phi(y)$. Since $\,y\to z$, a subsequence 
of the tangent directions of $\,\varGamma\hskip-3pt_y$ at $\,p_y$ tends to a 
limit, which must lie in 
$\,\mathrm{Ker}\hskip1.7pt\nabla\nh v_z\nh\subseteq\tzm\,$ (see 
Remark~\ref{submf}(ii)), so that, for the subsequence, $\,\lambda_y\nh\to0\,$ 
as $\,y\to z$. At the same time, according to Lemma~\ref{nvprl}(b\hs-iii) 
applied to the geodesic segment $\,\varGamma\subseteq K\,$ joining $\,z\,$ to 
$\,p_y$, each $\,\lambda_y$ is an eigen\-value of $\,\nabla\nh v_z$, and 
finiteness of the spectrum of $\,\nabla\nh v_z$ gives $\,\lambda_y\nh=0\,$ for 
all but finitely many $\,y\,$ in the subsequence. For such $\,y$, the equality 
$\,\nabla_{\!\dot x}v=[\lambda+\phi-\phi(y)]\hh\dot x\,$ reads 
$\,\nabla_{\!\dot x}v=\phi\hh\dot x$. Lemma~\ref{zrnlg}(i) now implies that 
$\,\phi=0\,$ at some interior point of each $\,\varGamma\hskip-3pt_y$ in 
question, contrary to the last paragraph.
\end{proof}
\begin{lemma}\label{nlsbs}Let\/ $\,M,g,v,Z,z,\phi,H\,$ be as in 
Theorem\/~{\rm\ref{nbzrs}}. If\/ $\,H\subseteq\tzm\,$ is a null subspace,  
$\,\phi(z)=0$, and\/ $\,\nabla\nh\phi_z\notin\nabla\nh v_z(\tzm)$, then\/ 
$\,Z\cap U'\nh=\hh\mathrm{exp}\hh_z[H\cap U\hh]\,$ for any sufficiently small 
star-shap\-ed neighborhood\/ $\,\,U\hs$ of\/ $\,0\,$ in\/ $\,\tzm\,$ mapped 
by\/ $\,\mathrm{exp}\hh_z$ dif\-feo\-mor\-phi\-cal\-ly onto a neighborhood\/ 
$\,\,U'$ of\/ $\,z\,$ in\/ $\,M\nh$.
\end{lemma}
\begin{proof}As $\,H\,$ is a null subspace, $\,H\subseteq H^\perp\nnh$. Thus, 
$\,H\smallsetminus V\nnh=$\hskip3.8pt{\rm\O} \hskip3ptfor 
$\,V=H\cap H^\perp\nnh$, and so, by in Lemma~\ref{tadir}, $\,\phi=0\,$ at all 
zeros of\/ $\,v\,$ near\/ $\,z$. On the other hand, for $\,\,U\nh,\,U'$ as 
above, Lemma~\ref{cstdm}(a) states that 
$\,K=\hh\mathrm{exp}\hh_z[H\cap U\hh]\,$ is a sub\-man\-i\-fold of $\,\,U'$ 
contained in both $\,\phi^{-1}(0)\,$ and the zero set 
$\,Z\,$ of $\,v$. For sufficiently small $\,\,U,\,U'$ and a sub\-man\-i\-fold 
$\,N\hs$ of $\,M\,$ chosen as in Lemma~\ref{intsb}, we now have 
$\,K\subseteq Z\cap U'\nh\subseteq N\nh\cap\phi^{-1}(0)$. Since 
$\,\phi(z)=0\,$ and 
$\,\nabla\nh\phi_z\notin\nabla\nh v_z(\tzm)$, Lemma~\ref{krnll}(b) implies 
that $\,\nabla\nh\phi_z$ is not orthogonal to the whole space 
$\,\mathcal{B}_z=\hs\mathrm{Ker}\hskip1.7pt\nabla\nh v_z$. Thus, 
$\,d\hh\phi_z$ is not identically zero on $\,\mathcal{B}_z$, and 
$\,H=\hs\mathcal{B}_z\cap\hs\mathrm{Ker}\hskip2.2ptd\hh\phi_z$ is a 
co\-di\-men\-\hbox{sion\hh-}\hskip.7ptone subspace of
$\,\mathcal{B}_z$. As $\,H=\tzk\,$ and $\,\tzn\nh=\hh\mathcal{B}_z$ (see 
Lemma~\ref{intsb}(ii)), $\,K\,$ is a co\-di\-men\-\hbox{sion\hh-}\hskip.7ptone 
sub\-man\-i\-fold of $\,N\nnh$, and the restriction of $\,\phi\,$ to $\,N\hs$ 
has a nonzero differential at $\,z$. Consequently, applying 
Lemma~\ref{nzero}(a) to $\,\beta=\phi$, we can make 
$\,N\nnh,\,U\,$ and $\,\,U'$ even smaller, so as to have 
$\,K=N\nh\cap\phi^{-1}(0)$, which proves our assertion since 
$\,K\subseteq Z\cap U'\nh\subseteq N\nh\cap\phi^{-1}(0)$.
\end{proof}

\section{Proof of Theorem~\ref{nbzrs}, case {\rm(\ref{cnd}.a)}}\label{ca}
Let $\,\,U\,$ be a star-shap\-ed neighborhood of $\,0\,$ in $\,\tzm\,$ 
mapped by $\,\mathrm{exp}\hh_z$ dif\-feo\-mor\-phi\-cal\-ly onto a 
neighborhood $\,\,U'$ of $\,z\,$ in $\,M\nh$, such that $\,\phi\ne0\,$ 
everywhere in $\,\,U'\nnh$. For every 
$\,x\in(Z\cap U'\hh)\smallsetminus\{z\}$, we denote by $\,L_x$ the tangent 
direction at $\,z\,$ of the geodesic segment $\,\varGamma\hskip-3pt_x\,$ 
joining $\,z\,$ to $\,x\,$ in $\,\,U'\nnh$. Then
\begin{equation}\label{lxs}
\varGamma\hskip-3pt_x\hskip6pt\mathrm{is\ null\ and}\hskip6ptL_x
\,\subseteq\,\mathcal{B}_z=\hs\mathrm{Ker}\hskip1.7pt\nabla\nh v_z\hskip6.5pt
\mathrm{for\ all}\hskip7ptx\in(Z\cap U'\hh)\smallsetminus\{z\},
\end{equation}
provided that $\,\,U\,$ and $\,\,U'$ are chosen small enough. In fact, 
$\,\varGamma\hskip-3pt_x\hs$ is null by Remark~\ref{phiez}. 
Lemma~\ref{zeros}(b) in turn shows that $\,v\,$ is tangent to 
$\,\varGamma\hskip-3pt_x$, and hence so is the covariant derivative of 
$\,v\,$ in the direction of $\,\varGamma\hskip-3pt_x$. Thus, each $\,L_x$ is 
contained in the eigen\-space of $\,\nabla\nh v_z$ for some eigen\-value 
$\,\lambda_x$. If, no matter how small one made $\,\,U\,$ and $\,\,U'\nnh$, 
the inclusion in (\ref{lxs}) failed to hold, there would exist a sequence, 
converging to $\,z$, of points $\,x\in(Z\cap U'\hh)\smallsetminus\{z\}\,$ with 
$\,\lambda_x\nh\ne0$. Passing to a subsequence, we would have 
$\,L_x\nh\to L\,$ for some line $\,L\,$ through $\,0\,$ in $\,\tzm\nh$. As 
$\,L\,$ would then be a radial limit direction of $\,Z\,$ at $\,z\,$ (cf.\ the 
end of Section~\ref{cl}), Remark~\ref{submf}(ii) with $\,\psi=v\,$ would 
imply that $\,L\subseteq\mathrm{Ker}\hskip1.7pt\nabla\nh v_z$, and so 
$\,\lambda_x\nh\to0$. Finiteness of the spectrum of $\,\nabla\nh v_z$ would 
now give $\,\lambda_x\nh=0\,$ for all but finitely many $\,x\,$ in the 
subsequence, contrary to how the subsequence was selected.

On the other hand, by Lemma~\ref{krnll}(a), 
\begin{equation}\label{nls}
\mathrm{both}\hskip6pt\mathcal{B}_z\nh
=\mathrm{Ker}\hskip1.7pt\nabla\nh v_z\hskip5.5pt\mathrm{and}
\hskip6ptH\subseteq\mathcal{B}_z\hskip5.5pt\mathrm{are\ null\ subspaces\ of}
\hskip6pt\tzm.
\end{equation}
If $\,\mathcal{B}_z$ is contained in $\,\mathrm{Ker}\hskip2.2ptd\hh\phi_z$, 
so that $\,H=\mathcal{B}_z$, (\ref{lxs}) yields (\ref{ncl}), with 
$\,C\cap H=H\,$ in view of (\ref{nls}), which, according to 
Remark~\ref{incls}, proves Theorem~\ref{nbzrs} when 
$\,\mathcal{B}_z\nh\subseteq\hh\mathrm{Ker}\hskip2.2ptd\hh\phi_z$.

From now on we therefore assume that $\,\mathcal{B}_z$ is {\it not\/} 
contained in $\,\mathrm{Ker}\hskip2.2ptd\hh\phi_z$. Thus, $\,H\,$ is a 
co\-di\-men\-\hbox{sion\hh-}\hskip.7ptone subspace of $\,\mathcal{B}_z$, 
and $\,K=\hh\mathrm{exp}\hh_z[H\cap U\hh]\,$ is a 
co\-di\-men\-\hbox{sion\hh-}\hskip.7ptone sub\-man\-i\-fold of 
$\,N\nh=\hh\mathrm{exp}\hh_z[\hh\mathcal{B}_z\cap U\hh]$, while the 
restriction of $\,\phi\,$ to $\,N\hs$ has a nonzero differential at $\,z$. In 
addition, by Lemma~\ref{cstdm}(a) and (\ref{nls}), $\,\phi=\phi(z)\,$ 
everywhere in $\,K$. According to Lemma~\ref{nzero}(a) for 
$\,\beta=\phi-\phi(z)$, making $\,\,U\,$ and $\,\,U'$ even smaller if 
necessary, we can ensure that $\,\phi\ne\phi(z)\,$ everywhere in 
$\,N\nh\smallsetminus K$. This shows that no zero $\,x\,$ of $\,v\,$ lies in 
$\,N\nh\smallsetminus K$, for if one did, Lemma~\ref{zrnlg}(ii) and (\ref{lxs}) 
would give $\,\phi=\phi(z)\,$ somewhere in 
$\,\varGamma\hskip-3pt_x\smallsetminus\{z\}\subseteq N\nh\smallsetminus K$. 
In other words, we again have (\ref{ncl}), with $\,C\cap H=H\,$ (cf.\ 
(\ref{nls})), which, in view of Remark~\ref{incls}, proves Theorem~\ref{nbzrs} 
in case (\ref{cnd}.a).

\section{Proof of Theorem~\ref{nbzrs}, case {\rm(\ref{cnd}.b)}}\label{cb}
We are free to assume that 
$\,\mathcal{B}_z\nh=\mathrm{Ker}\hskip1.7pt\nabla\nh v_z$ is not a null 
subspace of $\,\tzm\nh$. Namely, if $\,\mathcal{B}_z\nh\subseteq\tzm\,$ is a 
null subspace, then so is $\,H\subseteq\mathcal{B}_z$, and the assertion of 
Theorem~\ref{nbzrs} is immediate from Lemma~\ref{nlsbs}, with $\,C\cap H=H\,$ 
since $\,H\,$ is null.

Let us choose the sub\-man\-i\-folds $\,N\hs$ and $\,K\,$ of $\,M\,$ as in 
Lemmas~\ref{intsb} and~\ref{cstdm}, with $\,\,U\,$ small enough so as to 
ensure that $\,K\subseteq N\nnh$. Such $\,\,U\,$ must exist since, by 
Lemma~\ref{cstdm}(a), $\,K\,$ is contained in zero set $\,Z\,$ of $\,v$, while 
all zeros of $\,v\,$ close to $\,z\,$ lie in $\,N\nnh$. 
Lemmas~\ref{cstdm}(a),\hs(d) and~\ref{intsb}(ii) imply that, in fact, not only 
$\,v_x\nh=0$, but also  $\,\txn\nh=\hh\mathrm{Ker}\hskip1.7pt\nabla\nh v_x$ 
whenever $\,x\in K$. We may now use a local trivialization of $\,\tm\,$ on 
$\,\mathrm{exp}\hh_z(U)\,$ to identify each tangent space $\,\txm\nh$, for 
$\,x\in N\nnh$, with $\,\tzm\nnh$, in such a way that $\,\tzm\hs$ itself 
remains unchanged. (One could for instance use the identifications provided by 
parallel transports along geodesics emanating from $\,z$.) This allows us to 
treat $\,2v\,$ as a vec\-tor-val\-ued function $\,f:N\to\tzm\nh$.

The hypotheses of Theorem~\ref{tgdir} are now satisfied by our $\,K,N\nnh,z$, 
the vector space $\,\mathcal{T}\nnh=\tzm\hs$ with $\,\langle\,,\rangle=g_z$ 
and $\,f\,$ as above, $\,\phi:N\to\bbR\,$ obtained by restricting to $\,N\hs$ 
the function in (\ref{lvg}), and 
$\,Y\nnh=\hh\mathrm{exp}\hh_z[\hh C\cap H\cap U\hh]$, where 
$\,H=\hs\mathrm{Ker}\hskip1.7pt\nabla\nh v_z
\cap\hs\mathrm{Ker}\hskip2.2ptd\hh\phi_z\nh\subseteq\tzm\,$ and 
$\,C=\{u\in\tzm:g_z(u,u)=0\hs\}\,$ is the null cone, provided that one 
replaces $\,N\nnh,K\,$ and $\,\,U\,$ with suitable smaller neighborhoods of 
$\,z\,$ in $\,N\hs$ or $\,K\,$ and $\,0\,$ in $\,\tzm\nh$.

Specifically, $\,d\hh\phi_z$ is not identically zero on 
$\,\tzn\nh=\hh\mathrm{Ker}\hskip1.7pt\nabla\nh v_z$ since $\,\phi(z)=0\,$ and 
$\,\nabla\nh\phi_z\notin\nabla\nh v_z(\tzm)\,$ by (\ref{cnd}.b), and so, in 
view of Lemma~\ref{krnll}(b), $\,\nabla\nh\phi_z$ is not orthogonal to all of 
$\,\mathrm{Ker}\hskip1.7pt\nabla\nh v_z$. Next, $\,Y\hs$ is a quadric of the 
required kind due to the very definition of a quadric, in the lines preceding 
Lemma~\ref{quadr}, with the r\^ole of $\,\varPsi\,$ played here by the 
restriction of $\,\mathrm{exp}\hh_z$ to $\,H\nnh\cap U\nh$, which sends 
$\,H\nnh\cap U\hs$ into $\,P\nh=N\cap\phi^{-1}(0)\,$ according to 
Lemma~\ref{cstdm}(a), and is a dif\-feo\-mor\-phism for dimensional reasons. 
Condition (a) in Theorem~\ref{tgdir} holds in turn due to the assumption about 
$\,\mathcal{B}_z\nh=\tzn\hs$ made at the beginning of this section, condition 
(b) follows since $\,K=\hh\mathrm{exp}\hh_z[\hh V\nnh\cap U\hh]\,$ for 
$\,V=H\cap H^\perp$ (see Lemma~\ref{cstdm}), and (c) is immediate from 
Lemma~\ref{cstdm}(a) (which states that $\,v=0$, and hence $\,f=0$, on $\,Y$), 
combined with the equality $\,\txn\nh=\hh\mathrm{Ker}\hskip1.7pt\nabla\nh v_x$ 
for $\,x\in K$, established above (which amounts to $\,d\hskip-.8ptf=0\,$ 
everywhere in $\,K$). Lemma~\ref{intsb}(iv) now implies that the left-hand 
side in (d) equals the Hess\-i\-an at $\,z\,$ of the function 
$\,y\mapsto\langle w,f(y)\rangle\,$ on $\,N\nnh$.

In view of (\ref{cnd}.b), the assertion of Theorem~\ref{tgdir}(ii) amounts to 
the assumption of Lemma~\ref{tadir}, which now implies that all zeros of 
$\,v\,$ close to $\,z\,$ lie in $\,P\nh=N\cap\phi^{-1}(0)$. By 
Theorem~\ref{tgdir}(i), they must lie in $\,Y\hs$ as well, and so (\ref{ncl}) 
holds for sufficiently small $\,\,U\,$ and $\,\,U'\nnh$. Combined with 
Remark~\ref{incls}, this proves Theorem~\ref{nbzrs} in case (\ref{cnd}.b).

\section{Proof of Theorem~\ref{tusbm}}\label{cp}
Let us fix a point $\,z\in Z$. We denote by $\,\phi\,$ the function in 
(\ref{lvg}), by $\,H\,$ the subspace 
$\,\hs\mathrm{Ker}\hskip1.7pt\nabla\nh v_z
\cap\hs\mathrm{Ker}\hskip2.2ptd\hh\phi_z$ of $\,\tzm\nh$, and by $\,\tzz\,$ 
(or, $\,b_z$) the tangent space (or, respectively, the second fundamental 
form) at $\,z\,$ of the connected component of $\,Z\,$ containing $\,z$,
provided that $\,z\,$ is not a singular point of $\,Z$. Three cases are 
possible:
\begin{enumerate}
  \def\theenumi{{\rm\alph{enumi}}}
\item[{\rm(i)}] neither (\ref{cnd}.a) nor (\ref{cnd}.b) holds at $\,z$,
\item[{\rm(ii)}] $z\,$ satisfies (\ref{cnd}.a) or (\ref{cnd}.b) and the 
metric $\,g_z$ restricted to $\,H\,$ is not sem\-i\-def\-i\-nite,
\item[{\rm(iii)}] $z\,$ satisfies (\ref{cnd}.a) or (\ref{cnd}.b) and $\,g_z$ 
is sem\-i\-def\-i\-nite on $\,H$.
\end{enumerate}
In case (i), Theorem~\ref{escnf} allows us to change the metric 
con\-for\-mal\-ly so as to make $\,v\,$ a Kil\-ling field for the new 
metric $\,g\hh'$ on some nor\-mal-co\-or\-di\-nate neighborhood $\,\,U'$ of 
$\,z\,$ in $\,(M,g\hh'\hh)$. Assertion (a) in Theorem~\ref{tusbm} is now 
immediate, as $\,\mathrm{exp}\hh_z$ (corresponding to $\,g\hh'$) sends 
short line segments emanating from $\,0\,$ in $\,\tzm\,$ onto 
$\,g\hh'\nnh$-ge\-o\-des\-ics, and so the local flow of $\,v\,$ corresponds 
via $\,\mathrm{exp}\hh_z$ to the linear local flow near $\,0\,$ in 
$\,\tzm\nh$, generated by $\,\partial\hh v_z$ (notation of Section~\ref{dh}).

Consequently, in case (i), 
$\,\tzz=\hh\mathrm{Ker}\hskip2.7pt\partial\hh v_z\nh
=\hh\mathrm{Ker}\hskip1.7pt\nabla\nh v_z$. Also, since the 
$\,g\hh'\nnh$-Kil\-ling field $\,v\,$ has zero $\,g\hh'\nnh$-di\-ver\-gence, 
$\,\nabla\nh v_x\nh=\partial\hh v_x$ is trace\-less at every $\,x\in Z\,$ near 
$\,z$. Thus, $\,\phi=0\,$ at all points of $\,Z\,$ close to $\,z$. As a result, 
the co\-di\-men\-sion of $\,\tzz\,$ in $\,\tzm\,$ is even 
(Lemma~\ref{krnll}(b)), while $\,b_z$ is a tensor multiple of the metric due 
to Lemma~\ref{tgapt}(ii), the al\-read\-y\nh-es\-tab\-lish\-ed assertion (a) 
of Theorem~\ref{tusbm}, and (\ref{inv}).

Next, in cases (ii) and (iii), Theorem~\ref{nbzrs} clearly implies (b) in 
Theorem~\ref{tusbm}, with $\,g\hh'\nh=g$, while Lemma~\ref{cstdm}(a) shows 
that $\,\phi=\phi(z)\,$ at all points of $\,Z\,$ close to $\,z$. Combined 
with Remarks~\ref{semdf} and~\ref{snglr}, this gives the description of the 
singular subset $\,\Delta\,$ required in Theorem~\ref{tusbm}. Thus, in case 
(ii) (or, (iii)), $\,z\,$ is a singular (or, respectively, nonsingular) point 
of $\,Z$.

Consider now case (iii). In view of Theorem~\ref{nbzrs}, $\,\tzz\,$ is the 
null\-space of $\,H$, that is, $\,\tzz=H\cap H^\perp\nnh$, while 
$\,b_z\nh=0\,$ by Lemma~\ref{tgapt}(ii) and, as noted above, $\,Z\,$ has no 
singularities near $\,z$. It follows now that case (iii) represents an open 
condition, or, in other words, we will still have (iii) after $\,z\,$ has been 
replaced with any nearby point $\,x\in Z$. In fact, case (ii) for such $\,x\,$ 
cannot occur since they are nonsingular in $\,Z$. To exclude case (i) for 
them, note that (iii), for $\,z$, has two subcases: (\ref{cnd}.a) and 
(\ref{cnd}.b). In the former, (\ref{cnd}.a) obviously remains valid at 
nearby points. The latter subcase amounts in turn to assuming that 
$\,\phi(z)=0\,$ and $\,\mathrm{Ker}\hskip1.7pt\nabla\nh v_z$ is not 
contained in $\,\mathrm{Ker}\hskip2.2ptd\hh\phi_z$ (see 
Lemma~\ref{krnll}(b)). By Lemma~\ref{cstdm}(a),\hs(c),\hs(d), these 
assumptions will still hold when $\,z\,$ is replaced with any nearby 
$\,x\in K$, which, by Theorem~\ref{nbzrs}, are the same points as nearby 
$\,x\in Z$. Consequently, points $\,x\in Z\,$ near $\,z\,$ cannot represent 
case (i).

Thus, in case (iii), due to its o\-pen-con\-di\-tion property, the 
equalities $\,\tzz=H\cap H^\perp$ and $\,b_z\nh=0\,$ imply that the 
intersection of $\,Z\,$ with some neighborhood of $\,z\,$ is a null totally 
geodesic sub\-man\-i\-fold of $\,(M,g)$. The proof of Theorem~\ref{tusbm} is 
now complete.
\begin{remark}\label{opnns}As we just saw, case (iii) constitutes an open 
condition in the set $\,Z\,$ of all zeros of $\,v$. By Theorem~\ref{escnf}, 
the same is true of case (i). Not so, however, in case (ii): according to 
Theorem~\ref{nbzrs} and Remark~\ref{snglr}(b), in every neighborhood of a 
point $\,z\in Z\,$ representing case (ii), there exist points of $\,Z\,$ which 
are nonsingular, and hence, as we saw above, must correspond to case (i) or 
case (iii).
\end{remark}
\begin{remark}\label{lorcs}In the Lo\-rentz\-i\-an case, Theorem~\ref{tusbm} 
can obviously be rephrased so as to reflect the fact that null 
sub\-man\-i\-folds can be at most \hbox{one\hh-}\hskip0ptdi\-men\-sion\-al, 
while the only singularities of the zero set that may occur are those 
associated with null cones in Lo\-rentz\-i\-an subspaces of the tangent space.
\end{remark}
\begin{remark}\label{cfcir}Theorem~\ref{tusbm} provides hardly any information 
about those connected components of $\,(Z\cap U')\smallsetminus\Delta\,$ which 
happen to be \hbox{one\hh-}\hskip0ptdi\-men\-sion\-al. For sub\-man\-i\-folds 
$\,K\,$ with $\,\dim K=1$, the property of being totally umbilical is nearly 
meaningless, as such $\,K\,$ always has it, except at points $\,x\in K\,$ at 
which $\,\tzk\,$ is a null subspace and the second fundamental form $\,b_x$ is 
nonzero. It is worth pointing out that 
\hbox{one\hh-}\hskip0ptdi\-men\-sion\-al connected components of 
$\,(Z\cap U')\smallsetminus\Delta\,$ need not, in general, be {\it 
con\-for\-mal circles}. (For a definition, see \cite{bailey-eastwood}.) In 
fact, a non-null geodesic $\,t\mapsto x(t)\,$ in a 
pseu\-\hbox{do\hs-}\hskip0ptRiem\-ann\-i\-an manifold is a con\-for\-mal 
circle if and only if $\,S(\dot x,\,\cdot\,)=0$, where $\,S\,$ is the Schouten 
tensor \cite[p.\ 217]{bailey-eastwood}. Let $\,u\,$ now be a Kil\-ling field 
with a nonempty discrete set $\,Y$ of zeros on a 
pseu\-\hbox{do\hs-}\hskip0ptRiem\-ann\-i\-an manifold $\,(N,h)$, the scalar 
curvature of which is nonzero at some point $\,y\in Y\nnh$. (Such $\,u\,$ 
obviously exist on \hbox{even\hh-}\hskip0ptdi\-men\-sion\-al standard 
spheres.) Extending $\,u\,$ trivially to a Kil\-ling field $\,v\,$ on the 
product manifold $\,(M,g)=(\bbR,dt^2)\times(N,h)$, we see that the geodesic 
$\,\bbR\ni t\mapsto(t,y)\,$ forms a connected component of the zero set of 
$\,v$, while $\,\mathrm{Ric}\hh(\dot x,\,\cdot\,)=0$, and hence 
$\,S(\dot x,\,\cdot\,)\ne0\,$ due to the definition of $\,S\,$ in 
Section~\ref{cv}. Thus, the geodesic in question is not a con\-for\-mal circle.
\end{remark}

\section*{References}


\end{document}